\documentclass[a4paper,oneside,reqno]{amsart}

\calclayout

\usepackage{etoolbox}

\newtoggle{default}\toggletrue{default}
\newtoggle{birk}\togglefalse{birk}
\newtoggle{birk_t2}\togglefalse{birk_t2}
\newtoggle{siam}\togglefalse{siam}
\newtoggle{springer}\togglefalse{springer}

\newtoggle{endwithsymbol}\togglefalse{endwithsymbol}

\toggletrue{default}
\toggletrue{endwithsymbol}

\usepackage[utf8]{inputenc}
\usepackage[english]{babel}
\usepackage{microtype} 

\usepackage{amsmath,amssymb,amsthm}
\usepackage{thmtools} 

\usepackage{hyperref} 
\hypersetup{hidelinks}
\usepackage[nameinlink,sort]{cleveref} 
\usepackage{orcidlink}

\usepackage{mathtools} 
\usepackage{upgreek}
\usepackage{accents} 
\usepackage{stmaryrd} 

\usepackage{tikz} 
\usetikzlibrary{cd}
\usetikzlibrary{calc}

\usepackage{dsfont} 
\usepackage{pifont} 

\usepackage{ifthen}
\usepackage{enumitem}
\setlist[itemize]{labelindent=\parindent,leftmargin=*,itemsep=3pt,topsep=5pt}
\setlist[enumerate]{label=\textup{(\roman{*})},itemsep=3pt,labelindent=0pt,leftmargin=*}

\usepackage{nicematrix}

\usepackage{graphicx} 
\usepackage{subcaption}

\usepackage{booktabs} 
\usepackage{esint} 

\usepackage{bbm}

\usepackage{color}

\usepackage{pgfpages} 
\usepackage{pgfplots} 
\pgfplotsset{compat=1.13} 

\definecolor{NathanaelColor}{rgb}{0.2,0.0,0.6}

\newcommand{\R}{\mathbb{R}} 
\newcommand{\C}{\mathbb{C}} 

\newcommand{\1}{\mathds{1}}

\newcommand{\Cc}[1][\infty]{\mathrm{C}_{\mathrm{c}}\ifthenelse{\equal{#1}{}}{}{^{#1}}}
\newcommand{\Lp}[2][]{\mathrm{L}_{#2\ifthenelse{\equal{#1}{}}{}{,#1}}} 

\newcommand{\iu}{\mathrm{i}} 
\newcommand{\e}{\mathrm{e}}
\newcommand{\dd}{\mathrm{d}} 

\newcommand{\co}{\mathrm{c}}

\DeclareMathOperator{\ran}{ran}

\DeclareMathOperator{\spt}{spt}

\DeclareMathOperator{\dom}{dom}
\DeclareMathOperator{\diag}{diag}
\DeclareMathOperator{\dist}{dist}

\DeclareMathOperator{\grad}{grad}
\DeclareMathOperator{\dive}{div}
\DeclareMathOperator{\curl}{curl}

\DeclarePairedDelimiterX{\dset}[2]{\{}{\}}{#1\,\delimsize\vert\,\mathopen{} #2}
\DeclarePairedDelimiterX{\scprod}[2]{\langle}{\rangle}{#1,#2}

\renewcommand{\Re}{\operatorname{Re}}
\renewcommand{\Im}{\operatorname{Im}}

\theoremstyle{plain}
\newtheorem{theorem}{Theorem}[section]
\newtheorem{lemma}[theorem]{Lemma}
\newtheorem{proposition}[theorem]{Proposition}

\newtheorem{hypothesis}[theorem]{Hypothesis}

\iftoggle{endwithsymbol}{%
    \declaretheorem[style=definition,sibling=theorem,qed=\ding{169}]{definition}
    \declaretheorem[style=definition,sibling=theorem,qed=\ding{169}]{example}
    \declaretheorem[style=definition,sibling=theorem,qed=\ding{169}]{problem}
    \declaretheorem[style=definition,sibling=theorem,qed=\ding{169}]{assumption}
    \declaretheorem[style=definition,numbered=no,qed=\ding{169}]{claim}

    \declaretheorem[style=remark,sibling=theorem,qed=\ding{169}]{remark}
  }%
  {%
    \theoremstyle{definition}
    
    \newtheorem{example}[theorem]{Example}

    \theoremstyle{remark}
    \newtheorem{remark}[theorem]{Remark}
  }%

\begin{document}

\title[Stability for hyperbolic equations]{Block operator matrix techniques for stability properties of hyperbolic equations} 

%

%
%

\author[M.~Waurick]{Marcus Waurick\,\orcidlink{0000-0003-4498-3574}}
\email{marcus.waurick@math.tu-freiberg.de}

\address{TU Bergakademie Freiberg \\
  Institute of Applied Analysis \\
  Akademiestrasse 6 \\
  D-09596 Freiberg \\
  Germany}

\date{\today}
\dedicatory{}



\ifboolexpr{togl{birk_t2} or togl{birk}}{%
}{%
}%


\begin{abstract} Inspired by recent developments in the theory of stability results in the context of certain wave type phenomena, we discuss abstract damped hyperbolic type equations given in a block operator matrix form with regards to asymptotic behaviour of their solutions. Under mild conditions on the operators involved we provide criteria establishing strong or semi-uniform stability. In the particular case of Maxwell's equations, these criteria are implied under mild regularity conditions of the underlying domain causing spatial derivative operators satisfy certain compact embedding conditions and rather minimal assumptions on the damping conductivity. These assumptions improve on both regularity as well as on the structural requirements for the conductivity previously available in the literature.  \end{abstract}

\ifboolexpr{togl{default} or togl{birk_t2} or togl{birk}}{\maketitle}{}%

 \section*{Acknowledgments}
  The author thanks Sebastian Franz, Rainer Picard and Sascha Trostorff for useful discussions particularly concerning the decisive result in  \Cref{sec:crs}.
\section{Introduction}\label{sec:intro}

Quite recently, in \cite{EKL24,NS25}, the stability, that is, the time-asymptotic decay of solutions, of the following problem class has been considered. Let $H_0,H_1$ be Hilbert spaces, $\alpha=\alpha^*\in L(H_0)$, $\beta=\beta^* \in L(H_1)$, $\gamma\in L(H_0)$ be bounded linear operators. Furthermore, let $C\colon \dom(C)\subseteq H_0\to H_1$ be closed and densely defined and let $C^*$ be the corresponding Hilbert space adjoint. Then we consider the following equation:
\begin{equation}\label{eq:absMax}
  \big(\partial_t \begin{pmatrix} \alpha & 0 \\ 0 & \beta \end{pmatrix}U\big)'(t)= -\big( \begin{pmatrix} \gamma & 0 \\ 0 & 0 \end{pmatrix}+ \begin{pmatrix} 0 & -C^* \\ C & 0 \end{pmatrix}\big)U(t)
\end{equation}
subject to suitable initial conditions at $0$ in (a subset of) $H_0\times H_1$. If $C$ has closed range and $\Re \gamma\geq c>0$, that is, \emph{full damping}, and initial conditions are chosen from $H_0\times\beta^{-1}\ran(C)$, then the above solutions decay exponentially. This is implicitly contained in \cite{Tro15,DIW24} (see also \cite[Chapter 11]{STW22} for the context of evolutionary equations) and in the present form has been considered in \cite{EKL24}. In both approaches a suitable change of variables yields the corresponding result. The former uses a frequency domain formulation and the latter works more directly with the underlying energy. In the companion paper \cite{Wa26} of the present one, we have analysed the frequency domain approach more closely and provided our perspective to exponential stability. This reasoning and the simplifications and rationales developed along the way, however, uncover structural insights that are of substantial help regarding the situation, when $\Re \gamma\geq c>0$ does not hold any more. Thus, the aim of the present article is to use the insights gathered in \cite{Wa26} for studying the case of \emph{partial damping}. Only particular geometric set-ups then lead to exponential stability, see the classical \cite{Leb96} (and the references therein), involving the geometric control condition; this particular avenue is not followed up here as these results adhere to explicit solution representations that cannot be expected on the level of abstraction favoured in this article. This being said, it has been found that strong or semi-uniform stability results can be found for suitable initial data with a rather tame usage of knowledge of particular solutions. Indeed, for Maxwell's equations, such results have been appeared in \cite{NS25,Ell19}.  The aim of this paper is to simultaneously generalise these two sources in that weaker sufficient conditions for (strong) stability are provided and a precise functional analytic condition is coined leading to semi-uniform stability. 

In order to present the differences to these approaches right away, we shall specialise the above setting to the example situation of Maxwell's equations. For this let $\Omega\subseteq \R^3$ be open, bounded, connected and Lipschitz, $\varepsilon, \mu $ are self-adjoint matrix-valued coefficients Lipschitz continuous and strictly positive definite uniformly on $\Omega$. In the above setting, $C^*=\curl$, the distributional $\curl$-operator in $L_2(\Omega)^3=H_0=H_1$ with maximal domain and $C$ is the same obtained by the closure of $\curl$ restricted to smooth compactly supported vector fields. Next, $\alpha=\varepsilon$ and $\beta=\mu$, where $\varepsilon$ is the dielectricity and $\mu$ the magnetic permeability. The (partial) damping is then introduced via the electric conductivity $\sigma\in L_\infty(\Omega)^{3\times 3}$ assumed to satisfy $\Re \sigma\geq 0$ rendering the corresponding semi-group generator to be m-dissipative.

We find the following conditions on \textbf{strong stability}, that is, mild solutions of \cref{eq:absMax} tend to $0$ in norm as $t\to\infty$ for (suitable) initial data:
\begin{enumerate}
\item[\cite{Ell19}] if $\sigma\in W_\infty^1(\Omega)^{3\times 3}$ is non-negative; $\sigma$ is positive definite on an open non-empty subset $\omega\subseteq \Omega$ and $\sigma=0$ on $\Omega\setminus\omega$. Initial data needs to be divergence-free.
\item[\cite{NS25}${}^{\text{st}}$]  if $\sigma\in L_\infty(\Omega;\R)$, $\sigma\geq 0$ and with $\Omega_0 \coloneqq \bigcup\{U; U\subseteq \Omega\text{ open }, \sigma=0\text{ on }U\}\neq\emptyset$; $\Omega_+\coloneqq \Omega\setminus \overline{\Omega_0}\neq \emptyset$. Moreover, $\partial\Omega,  \partial\Omega_0, \partial\Omega_+$ are Lipschitz and $\partial_0\Omega\cap \partial\Omega$ is a Lipschitz submanifold of $\partial\Omega$. Initial data need to be perpendicular to the stationary states of \cref{eq:absMax}.
\end{enumerate}

Furthermore, in the latter reference we can also find a set-up for \textbf{semi-uniform stability}, that is, solutions decay with a prescribed rate $f\colon [0,\infty)\to [0,\infty)$ (with $f(t) \to 0\text{ as }t\to\infty$) unifomly for initial data in the generator domain, if
\begin{enumerate}
\item[\cite{NS25}${}^{\text{su}}$] in addition to \cite{NS25}${}^{\text{st}}$, we have $\partial\Omega$ to be connected. There are finitely many connected open and non-empty sets $\Omega_{+,k}\subseteq \Omega$, $k\in \{1,\ldots,n\}$, with pairwise disjoint closures. Next, $\overline{\Omega_{+,k}}$ either intersects with $\partial\Omega$ nowhere or on a set of positive relative measure. Finally $\sigma\geq \sigma_->0$ on $\Omega_+=\bigcup_{k\in \{1,\ldots,n\}}\Omega_{+,k}$.
\end{enumerate}

In the present article we improve the above results in the following situations.

\subsection*{Strong stability} We will show that for all non-stationary initial data, strong stability holds for $\sigma\in L_\infty(\Omega)^{3\times 3}$ satisfying  for some $\omega\subseteq \Omega$ open and $c>0$, $\Re \sigma \geq c>0$ on $\omega$ and $\sigma=\sigma^*\geq 0$ on $\Omega\setminus\omega$. Thus, we confirm Eller's conjecture in \cite{Ell19} on the non-optimality of the conditions on $\omega$ and improve on the regularity assumption. Moreover, we provide a more general geometric set-up than the one considered in \cite{NS25}${}^{\text{st}}$ and allow for matrix-valued and certain non-selfadjoint settings. Thus, we have found a common proper generalisation of the available criteria in the strongly stable situation.

\subsection*{Semi-uniform stability} For semi-uniform stability, we require $D\subseteq \Omega$ open with continuous boundary and such that $L_2(\Omega)=L_2(D)\oplus L_2(\Omega\setminus\overline{D})$ and $\sigma = \tilde{\sigma}\1_{D}$, where $\tilde{\sigma}\in L_\infty(D)^{3\times 3}$ with $\Re \tilde{\sigma}\geq c>0$. Moreover, we require the geometric compatibility condition
\begin{equation}\label{eq:closedD}
   \1_{D} [\grad[H_0^1(\Omega)]]\subseteq L_2(D)^3\text{ closed.}
\end{equation}
If $\overline{D}\subseteq \Omega$ and $D$ is an $H^1$-extension domain, then $   \1_{D} [\grad[H_0^1(\Omega)]]= \grad[H^1(D)]$ and the right-hand side is closed. This particularly implies that in order to guarantee semi-uniform stability the precise geometric set-up as provided in \cite{NS25}${}^{\text{su}}$ is not necessary. We conjecture that the conditions in \cite{NS25}${}^{\text{su}}$ do imply \cref{eq:closedD}, but we have failed to confirm this yet.

We note that albeit the case of full damping implying exponential stability can be established using abstract means only, see \cite{EKL24,Wa26}, the situation for strong or semi-uniform stability requires closer attention to the particular set-up. In fact, the stability part itself only holds due to injectivity results, which themselves are consequences of a unique continuation principle valid for Maxwell's equations, see below and \cite{NW12} for the original reference. However, the abstract perspective also allows for more insight in that it is possible to provide counterexamples in general Hilbert spaces, thus confirming that there cannot be an abstract proof showing semi-uniform stability. We note in passing that the general Hilbert space technique can be used to reduce the analytical effort significantly and help to identify the places where particular computations are required.

The paper is organised in the following way. In \Cref{sec:wp} we introduce our basic assumptions and show well-posedness of the above problem whilst at the same time showing that we may assume without loss of generality that  $\alpha=1$ and $\beta=1$. In \Cref{sec:3b3} we provide the main structural insight from \cite{Wa26} in that we may think of the above equation as an equation with both $C$ and $C^*$ being one-to-one and onto with additional infinite-dimensional condition involving $\gamma$. \Cref{sec:stsu} summarises the main stability theorems from the literature we embark to employ. Moreover, we specify the abstract conditions needed to obtain strong and semi-uniform stability for \cref{eq:absMax} and prove the corresponding theorems. For this some auxiliary statements concerning closed range results are needed, which are provided in \Cref{sec:clr}. In this section, we also provide counterexamples necessitating more concrete set-ups warranting the closedness of certain operators. For this reason, we introduce the main application, Maxwell's equations with partial damping in \Cref{sec:apptoMax}. We may prove strong stability right away and then address semi-uniform stability. This in turn requires revisiting an inequality from \cite{PPTW21} (also provided in a different form in \cite{NS25}). In a $d$-dimensional setting not using the underlying complex structure, we prove this inequality in \Cref{sec:crs}. We conclude the paper with a summary and an outlook.

\section{Changing the variables and the generation theorem for  \cref{eq:absMax}}\label{sec:wp}

The standing abstract hypothesis (which is going to be refined later) reads as follows.

\begin{hypothesis}\label{hyp:fne0} Let $H_0$, $H_1$ be Hilbert spaces, $\alpha,\gamma\in L(H_0)$, $\beta\in L(H_1)$, $C\colon \dom(C)\subseteq H_0\to H_1$ be densely defined and closed. 

Assume $\alpha=\alpha^*$ and $\beta=\beta^*$ and that there is $c>0$ such that  \[
\beta\geq c, \quad \alpha\geq c.
\]Moroever, $\Re \gamma\coloneqq (1/2)(\gamma+\gamma^*)\geq 0$.
\end{hypothesis}

We quickly summarise the setting provided in \cite{Wa26} for showing well-posedness of \cref{eq:absMax}. Note that as $\alpha=\alpha^*\geq c$ and $\beta=\beta^*\geq c$, the corresponding square roots are topological isomorphisms and $\tilde{U}\coloneqq \diag(\sqrt{\alpha},\sqrt{\beta})U$ satisfies the equation
\begin{equation}\label{eq:absMax2}
\tilde{U}'(t)= -\big( \begin{pmatrix} \sqrt{\alpha}^{-1}\gamma\sqrt{\alpha}^{-1} & 0 \\ 0 & 0 \end{pmatrix}+ \begin{pmatrix} 0 & -\sqrt{\alpha}^{-1}C^*\sqrt{\beta}^{-1} \\ \sqrt{\beta}^{-1}C\sqrt{\alpha}^{-1} & 0 \end{pmatrix}\big)\tilde{U}(t).
\end{equation}
A similar change of variables uniquely provides a solution $U$ if $\tilde{U}$ satisfying \cref{eq:absMax2} is provided.

Since $(\sqrt{\beta}^{-1}C\sqrt{\alpha}^{-1} )^*=\sqrt{\alpha}^{-1}C^*\sqrt{\beta}^{-1}$ and $\Re \sqrt{\alpha}^{-1}\gamma\sqrt{\alpha}^{-1}\geq 0$ as $\Re \gamma\geq0$  it is not difficult to see that the operator
\[
-\big( \begin{pmatrix} \sqrt{\alpha}^{-1}\gamma\sqrt{\alpha}^{-1} & 0 \\ 0 & 0 \end{pmatrix}+ \begin{pmatrix} 0 & -\sqrt{\alpha}^{-1}C^*\sqrt{\beta}^{-1} \\ \sqrt{\beta}^{-1}C\sqrt{\alpha}^{-1} & 0 \end{pmatrix}\big)
\]
is m-dissipative in the Hilbert space $H_0\times H_1$ (it is a sum of the dissipative and bounded $ \begin{pmatrix} \sqrt{\alpha}^{-1}\gamma\sqrt{\alpha}^{-1} & 0 \\ 0 & 0 \end{pmatrix}$ and the skew-selfadjoint $\begin{pmatrix} 0 & -\sqrt{\alpha}^{-1}C^*\sqrt{\beta}^{-1} \\ \sqrt{\beta}^{-1}C\sqrt{\alpha}^{-1} & 0 \end{pmatrix}$). Thus, by the Lumer--Phillips theoerm, it generates a $C_0$-semigroup of contractions and hence given initial data from $H_0\times H_1$, \cref{eq:absMax2} admits a unique mild solution provided by the semi-group generated by $A$ applied to the initial data. In particular, we get unique existence of continuous solutions of \cref{eq:absMax} satisfying the equation in an integrated sense.

The state space $H_0\times H_1$ can also be reduced to the following
\[
   H =H_0 \oplus \ran(\sqrt{\beta}^{-1}C\sqrt{\alpha}^{-1})=H_0\oplus \sqrt{\beta}^{-1}[\ran(C)],
\]yielding still the generation property for $A$:
\begin{theorem}\label{thm:sgg} Assume \Cref{hyp:fne0}. Then the operator 
\[
A\coloneqq -\begin{pmatrix} \sqrt{\alpha}^{-1}\gamma\sqrt{\alpha}^{-1} & 0 \\ 0 & 0 \end{pmatrix}- \begin{pmatrix} 0 & -\sqrt{\alpha}^{-1}C^*\sqrt{\beta}^{-1} \\ \sqrt{\beta}^{-1}C\sqrt{\alpha}^{-1} & 0 \end{pmatrix}
\]
is m-dissipative and, thus, generates a contraction semi-group on $H$ and on $\overline{\ran}(A)$.
\end{theorem}
\begin{proof}
Note that the provided operator generates a semi-group on $H_0\times H_1$ as the latter summand is skew-self-adjoint and the former is dissipative. The claim follows by observing that $\ran(A)\subseteq \overline{\ran}(A)\subseteq H$ and, hence, $A$ restricted to $H$ (or to $\overline{\ran}(A)$) generates a $C_0$-semigroup by \cite[ch. II sec. 2.3]{EN00}.
\end{proof}

\section{A 3-by-3 matrix representation}\label{sec:3b3}

With the perspective offered in the previous section, we may reduce complexity of the problem right away whilst still addressing the general situation. Moreover, it will be crucial to have closedness of the range of $C$ as well. We will gather these new assumptions in the next Hypothesis:
\begin{hypothesis}\label{hyp:fne0p} Assume \Cref{hyp:fne0}. Moreover, assume that $\alpha=1$ and $\beta=1$ and that $\ran(C)\subseteq H_1$ is closed.
\end{hypothesis}
The consequences of $\ran(C)\subseteq H_1$ being closed are plenty. For this we refer to \cite{PW26} (or the FA-toolbox in \cite{PZ20}) for a long list of applications. Since only a few of them are needed here, we gathered them in the following remark, which also contains the definions of $\iota_0,\kappa_0,\iota_1, \kappa_1$ used throughout.
\begin{remark}\label{rem:clr}
(a) By Banach's closed range theorem, see, e.g., \cite[Theorem IV.1.2]{Gol06}, $\ran(C)\subseteq H_1$ is closed if and only if $\ran(C^*)\subseteq H_0$ is closed. Thus, by \Cref{hyp:fne0p}, \emph{both} $\ran(C)\subseteq H_1$ and $\ran(C^*)\subseteq H_0$ are closed.

(b)  If \Cref{hyp:fne0p} holds, then $H_0$ and $H_1$ adhere to the direct orthogonal sum  decomposition, which in applications can be the Helmholtz-decomposition,
\[
   H_1 = \ran(C)\oplus \ker(C^*)\text{ and }H_0 = \ran(C^*)\oplus \ker(C).
\]
We introduce $\iota_1 \colon \ran(C)\hookrightarrow H_1$, the continuous embedding. $\iota_1^*\colon H_1\to \ran(C)$ is the (surjective) orthogonal projection. Similarly, denote by $\kappa_1\colon \ker(C^*)\hookrightarrow H_1$ the canonical embedding and $\iota_0\colon \ran(C^*)\hookrightarrow H_0$ and $\kappa_0\colon \ker(C)\hookrightarrow H_0$. It follows from the closed graph theorem that $\iota_0^*C\iota_1$ and $\iota_1C\iota_0$ are continuously invertible. Note that these operators are densely defined as the above abstract Helmholtz decomposition allows for
\[
\dom(C^*) = \ran(C)\cap \dom(C^*)\oplus \ker(C^*)\text{ and }\dom(C)= \ran(C^*)\cap\dom(C)\oplus \ker(C)
\]as $\ker(C^*)\subseteq \dom(C^*)$ and $\ker(C)\subseteq \dom(C)$.

(c) With the FA-toolbox, e.g., in \cite{PZ20}, a sufficient condition fo $C$ to have closed range is that $\dom(C)\cap \ker(C)^\bot \hookrightarrow H_0$ is compact; see also the \cite[Appendix]{DIW24} for the corresponding standard contradiction argument.
\end{remark}
Using \Cref{hyp:fne0p} and recalling $A$ from the previous section, we get
\[
A=  -\begin{pmatrix} \gamma  & 0 \\ 0 & 0 \end{pmatrix}- \begin{pmatrix} 0 & - C^* \\ C& 0 \end{pmatrix}.
\]
Next, we aim to use the abstract Helmholtz decomposition to get more detailed information about the behaviour of $A$. Note that this decomposition was already derived in \cite[Theorem 4.4]{Wa26} and successfully applied to obtain exponential stability result. For convenience we sketch the short proof.

\begin{theorem}[{{\cite[Theorem 4.4]{Wa26}}}]\label{thm:soldecom} Assume \Cref{hyp:fne0p}. Let $f\in H_0, g\in 
\ran(C)$, $z\in \C$ and let $(u,v)\in \dom(A)$, $v\in \ran(C)$.

Then the following conditions are equivalent:
\begin{enumerate}
\item[(i)] $(z-A)(u,v)=(f,g)$
\item[(ii)] With $u=\iota_0\iota^*_0 u+  \kappa_0\kappa^*_0 u$ and $v=\iota_1\iota^*_1 v$  we have  \[\Big(z\begin{pmatrix}1 & 0&0\\ 0 & 1& 0 \\ 0 & 0 & 1 \end{pmatrix}+
\begin{pmatrix} \iota_0^*\gamma \iota_0 & 0&\iota_0^*\gamma \kappa_0\\ 0 & 0& 0 \\ \kappa_0^*\gamma \iota_0 & 0 & \kappa_0^*\gamma \kappa_0 \end{pmatrix} + \begin{pmatrix}0 & -\iota_0^*C^*\iota_1 &0\\ \iota_1^*C \iota_0 & 0& 0 \\ 0 & 0 & 0 \end{pmatrix}\Big)\begin{pmatrix} \iota_0^*u \\ \iota_1^*v \\ \kappa_0^*u\end{pmatrix} = \begin{pmatrix} \iota_0^*f \\ \iota_1^*g \\ \kappa_0^*f\end{pmatrix}.\]
\end{enumerate}
\end{theorem}
\begin{proof}
Reading (i) line by line we get
\[
    z u + \gamma u -C^*v = f\text{ and } z v + Cu =g.
\]
Then we use the equations $u=\iota_0\iota^*_0 u+  \kappa_0\kappa^*_0 u$ and $v=\iota_1\iota^*_1 v$  and re-write the result in matrix form. We emphasise $(u,v)\in \dom(A)$ that yields $\iota_1^*v\in \dom(\iota_0^*C\iota_1)$ and $\iota_0^*u\in \dom(\iota_1^*C\iota_0)$ by \Cref{rem:clr}(b). The converse implication is equally straightforward.\end{proof}

\section{Strong and semi-uniform stability}\label{sec:stsu}

There is a well-established (and in fact still growing) machinery of results on non-exponential stability for $C_0$-semigroups. Here we have occasion to only use two well-known ones. The results we want to apply here are due to Batty--Duyckaerts as well as Arendt--Batty--Lyubich--Vu and offer a semi-uniform stability result and a criterion for strong stability in terms of the resolvent of the generator, respectively.

\begin{theorem}[Batty--Duyckaerts, \cite{BD08}]\label{thm:BD} Let $T$ be a bounded $C_0$-semigroup in a Banach space $X$ with generator $A$. If $\sigma(A)\cap \iu \R=\emptyset$, then there exists $f\colon [0,\infty)\to [0,\infty) $ with $f(t)\to 0$ as $t\to\infty$ such that
\[
   \forall x\in \dom(A)\colon \| T(t) x \|\leq f(t)\|x\|_{\dom(A)}.
\]
\end{theorem}

The ABLV-stability result, see \cite{AB88,LV88}, will also come in handy in the application part. We present a simplified version of it as follows.

\begin{theorem}[ABLV-Theorem]\label{thm:ABLV} Let $T$ be a bounded $C_0$-semigroup in a Hilbert space $H$ with generator $A$ with no eigenvalues on $\iu\R$ and assume that $\sigma(A)\cap \iu\R$ is countable. Then $T$ is \textbf{strongly stable}, that is, for all $x\in H$, then
\[
   \lim_{t\to\infty}\|T(t)x\|=0.
\]
\end{theorem}

Next, specify the assumption on $\gamma$.
\begin{hypothesis}\label{hyp:gammasus}
Let $\gamma\in L(H_0)$. Assume there exists a decomposition of $H_0$ into closed subspaces $U,U^{\bot}$, $\gamma_{U}\in L(U)$, $\gamma_{U^\bot}\in L(U^\bot)$ with $\Re {\gamma}_U \geq c$ for some $c>0$ and $\gamma_{U^\bot}=\gamma_{U^\bot}^*\geq 0$ such that for $x = x_{U}+x_{U^\bot}\in U\oplus U^\bot$ we have
\[
   \gamma x = {\gamma}_U x_U+{\gamma}_{U^\bot} x_{U^\bot}.
\]
\end{hypothesis}
\begin{remark}\label{rem:observgamma}
(a) Written in $2$-by-$2$ block operator matrix form with respect to $U\oplus U^\bot$, $\gamma$ satisfying \Cref{hyp:gammasus} admits the representation
\[
   \gamma = \begin{pmatrix} {\gamma}_U & 0 \\ 0 & {\gamma}_{U^\bot} \end{pmatrix}.
\]
Moreover, note that $\Re \gamma\geq 0$ regardless.

(b) Let $\gamma \in L(H_0)$. Then $\gamma$ satisfies \Cref{hyp:gammasus} if and only if $\gamma^*$ satisfies \Cref{hyp:gammasus}. In either case, we can use the same decomposition of $H_0$ and the same constant $c$.

(c) If $\gamma$ satisfies \Cref{hyp:gammasus}, then any non-zero $\lambda\in \R$ leads to $\iu \lambda \in \rho(\gamma)$. Indeed, using the respresentation in (a), we deduce
\[
   \iu \lambda + \gamma =  \begin{pmatrix} \iu \lambda+ {\gamma}_U & 0 \\ 0 & \iu \lambda+{\gamma}_{U^\bot} \end{pmatrix}.
\]Then the top left entry is invertible since $\Re  (\iu \lambda+ {\gamma}_U)\geq c$ and the bottom right one is invertible since $\gamma_{U^\bot}$ is self-adjoint and, thus, has only real spectrum.
\end{remark}
As a consequence of \Cref{thm:BD} and \Cref{thm:ABLV}, we need to study the resolvent of $A$ on the imaginary axis. For this, we study the resolvent problem and consider conditions that allow us to deduce that for all $\lambda\in \R$
\[
   \iu \lambda -A
\]is one-to-one and has closed range. The injectivity part can only be established with an assumption on the abstract side and a unique continuation property for particular applications. The surjectivity part requires two different approaches for $\lambda \neq 0$ and $\lambda =0$. We will show that $\iu \lambda-A$ is onto for non-vanishing $\lambda$ and for $\lambda=0$ we will reduce the question to the consideration of a particular operator the closed range property of which can be shown in the application part.

The case for $\lambda=0$ can be ignored, if one is interested in strong stability, only.

\begin{theorem}\label{thm:sts} Assume \Cref{hyp:fne0p} and \Cref{hyp:gammasus} and that $\dom(C)\cap \ker(C)^\bot \hookrightarrow H_0$ is compact. Assume that 
$\iu \lambda-A$ and $\iu \lambda-A^*$ are one-to-one for all $\lambda\in \R\setminus\{0\}$.

Then $A_0\coloneqq A|_{\ker(A)^\bot}^{\overline{\ran}(A)}$ generates a strongly stable $C_0$-semi-group.
\end{theorem}

\begin{theorem}\label{thm:sus} Assume  \Cref{hyp:fne0p} and  \Cref{hyp:gammasus} and that $\dom(C)\cap \ker(C)^\bot \hookrightarrow H_0$ is compact. If \begin{enumerate}
\item[(a)] $\iu \lambda-A$ and $\iu \lambda-A^*$ are one-to-one for all $\lambda\in \R\setminus\{0\}$;
\item[(b)] $\kappa_0^*\gamma\kappa_0$ has closed range, where $\kappa_0 \colon \ker(C)\hookrightarrow H_1$ is the canonical embedding,
\end{enumerate}
then $A_0\coloneqq A|_{\ker(A)^\bot}^{\ran(A)}$ generates a semi-uniformly stable $C_0$-semi-group.
\end{theorem}

The proofs of \Cref{thm:sus} and \Cref{thm:sts} will be carried out in two steps. The main task is establishing the closed range of $\iu \lambda -A$, for which we use results of the next section. We provide these results in the summary next. 

\begin{theorem}\label{thm:clrsum}  Assume  \Cref{hyp:fne0p} and  \Cref{hyp:gammasus} and that $\dom(C)\cap \ker(C)^\bot \hookrightarrow H_0$ is compact. Then
\begin{enumerate}
\item[(a)] $A$ has closed range if and only if $\kappa_0^*\gamma \kappa_0$ has closed range.
\item[(b)] For all $\lambda\in \R\setminus\{0\}$, $\iu \lambda-A$ has closed range.
\item[(c)] $\ker(A)=\ker(A^*)$.
\end{enumerate}
\end{theorem}
\begin{remark}
The statement in (c) merely requires \Cref{hyp:fne0p} and \Cref{hyp:gammasus} (and even these assumptions are too much, see \Cref{thm:kerneladj} below). Note that this result is provided in a setting related to Maxwell's equations in \cite[Lemma 3.5]{NS25} in disguise of the statement $\ran(A)^\bot=\ker(A)$ with a rather technical proof.
\end{remark}
We are now in a position to prove the abstract stability theorems.

\begin{proof}[Proofs of \Cref{thm:sts} and \Cref{thm:sus}] Note that dissipativity of $A$ carries over to the same property for $A_0$ and therefore the generated semi-group is a semi-group of contractions. The only difference in assumptions between the two theorems is (b) in \Cref{thm:sus}. For the proofs of the stability theorems, we will use \Cref{thm:ABLV} and \Cref{thm:BD}.

Assuming (a) only, we show strong stability first. Note that by construction, $A_0$ is one-to-one. Next, we will show  $\sigma(A_0)\cap \iu \R\subseteq \{0\}$. Then  \Cref{thm:sts} follows from \Cref{thm:ABLV}. 

Since $\ker(A)=\ker(A^*)$, by \Cref{thm:clrsum}(c), the Hilbert space admits the decomposition $\ker(A)\oplus \overline{\ran}(A)$. In particular, $\iu \lambda -A_0$ leaves $\overline{\ran}(A)$ invariant (and so does its adjoint, which coincides with $-\iu\lambda - (A^*)_0$). Thus, in the Hilbert space $\ker(A)\oplus \overline{\ran}(A)$, we can write $A$ in block operator matrix form as
\[
   \begin{pmatrix} A_0 & 0 \\ 0 & 0
   \end{pmatrix}.
\]
 Next, for $\lambda\in \R\setminus\{0\}$ as $\iu\lambda-A$ has closed range, so has $\iu\lambda -A_0$. To establish that $\iu\lambda-A_0$ is onto, it suffices to see that $-\iu \lambda -A_0^*$ is one-to-one, which is true by assumption. Hence, the closed graph theorem confirms that $\iu \lambda \in \rho(A_0)$. Thus,
$\sigma(A_0)\cap \iu \R\subseteq \{0\}$, as desired and \Cref{thm:sts} follows.

If, additionally, (b) from \Cref{thm:sus} holds, then $\ran(A)$ is closed by \Cref{thm:clrsum}. Hence, $A_0$ is invertible, by the closed graph theorem, and we infer $\sigma(A_0)\cap \iu \R=\emptyset$ and \Cref{thm:sus} follows with an application of \Cref{thm:BD}.
\end{proof}

Next, we come to the proof of \Cref{thm:clrsum} in the next section.

\section{Closed range statements for the system operator}\label{sec:clr}

Throughout this section, we adopt \Cref{hyp:fne0p} and \Cref{hyp:gammasus}. Moreover, assume that $\dom(C)\cap \ker(C)^\bot \hookrightarrow H_0$ is compact. We introduce for $\lambda\in \R$
\begin{equation}\label{eq:bl}B(\lambda)\coloneqq \Big(\iu\lambda \begin{pmatrix}1 & 0&0\\ 0 & 1& 0 \\ 0 & 0 & 1 \end{pmatrix}+
\begin{pmatrix} \iota_0^*\gamma \iota_0 & 0&\iota_0^*\gamma \kappa_0\\ 0 & 0& 0 \\ \kappa_0^*\gamma \iota_0 & 0 & \kappa_0^*\gamma \kappa_0 \end{pmatrix} + \begin{pmatrix}0 & -\iota_0^*C^*\iota_1 &0\\ \iota_1^*C \iota_0 & 0& 0 \\ 0 & 0 & 0 \end{pmatrix}\Big).
\end{equation}

\begin{lemma}\label{lem:gaminv} Assume \Cref{hyp:gammasus}. Let $\lambda \in \R\setminus\{0\}$. Then for all closed subspaces $V\subseteq H_0$, and corresponding embedding $\iota\colon V\hookrightarrow H_0$, the operator $\iu\lambda + \iota^*\gamma\iota$ is continuously invertible.
\end{lemma}
\begin{proof}
We start off with a general observation. Let $H$ be a Hilbert space and $\eta\in L(H)$. Then, for all $\lambda\in \R$ and $\alpha\in \R$ we obtain
 \[
     \Re   (\e^{\iu \alpha}(\iu \lambda+\eta))=-\lambda \sin \alpha + \cos \alpha \Re \eta - \sin{\alpha}\Im \eta = \cos\alpha\Re\eta - \sin\alpha(\lambda+ \Im \eta), 
 \]
  where we recall $2\Re\eta=\eta+\eta^*$ and $2\iu\Im \eta=\eta-\eta^*$. 
 
 Turning to our task, we treat the case $V=H_0$ first. Then $\iu\lambda+\gamma$ is continuously invertible by \Cref{rem:observgamma} (c). Moreover, we may write
 \[
    \iu\lambda+\gamma=  \begin{pmatrix}
   \iu\lambda +           {\gamma}_U & 0 \\ 0 & \iu \lambda+{\gamma}_{U^\bot} 
     \end{pmatrix}
 \]
 for some $ {\gamma}_U\in L(U)$  and $0\leq  {\gamma}_{U^\bot}^*= {\gamma}_{U^\bot}\in L(U^\bot)$ such that $\Re{\gamma}_U\geq c$ for some $c>0$. Next, for the general situation, let $V\subseteq H_0$ be a closed subspace and let us restrict to the case $\lambda>0$ ($\lambda<0$ will follows similar lines). Then, by our preliminary observation, we may choose $\alpha \in (-\pi/2,0)$ and so close to $0$ so that 
 \begin{align*}
 \Re   (\e^{\iu \alpha}(\iu \lambda+ {\gamma}_U ))& \geq (\cos \alpha) c -|\sin\alpha| (\lambda +\|\Im{\gamma}_U \|)>0, \text{ and }\\
 \Re   (\e^{\iu \alpha}(\iu \lambda+ {\gamma}_{U^\bot} )) &\geq (\cos \alpha) {\gamma}_{U^\bot} -(\sin \alpha) \lambda\geq  -(\sin \alpha) \lambda>0.
 \end{align*} Thus, we find $\tilde{c}>0$ such that
 \[
    \Re \e^{\iu \alpha} (\gamma+\iu \lambda)\geq \tilde{c}.
 \]
 It follows that $\Re \iota^*\e^{\iu \alpha} (\gamma+\iu \lambda)\iota\geq \tilde{c}$ so that $ \iota^*\e^{\iu \alpha} (\gamma+\iu \lambda)\iota$ is continuously invertible, see, e.g., \cite[Proposition 6.2.3(b)]{STW22}. Thus, also $ \iota^* (\gamma+\iu \lambda)\iota = \e^{-\iu \alpha}\iota^*\e^{\iu \alpha} (\gamma+\iu \lambda)\iota$ is continuously invertible.
\end{proof}

We require an elementary lemma about closed ranges:

\begin{lemma}\label{lem:SAT}
Let $H$ be a Hilbert space, $A\colon \dom(A)\subseteq H\to H$ closed, $T,S\in L(H)$ be topological isomorphisms. Then
\[
  \ran(A)\subseteq H\text{ closed }\iff \ran(SAT)\subseteq H\text{ closed}.
\]
\end{lemma}
\begin{proof}
Since $S^{-1}$ and $T^{-1}$ are topological isomorphisms, it suffices to show $\ran(SAT)$ is closed as long as $A$ is. For this, let $(y_n)_n\coloneqq (SATx_n)_n$ be a sequence in $\ran(SAT)$ convergent to some $y\in H$ for some $(x_n)_n$ in $\dom(SAT)$. Then $(S^{-1}y_n)_n \to S^{-1}y$ as $S^{-1}$ is continuous. Since $S^{-1}y_n = ATx_n \in \ran(A)$ and $\ran(A)\subseteq H$ is closed, we find $z \in \dom(A)$ such that $S^{-1} y = Az$. Thus, $y = SAz=SATT^{-1}z$ so that $T^{-1}z\in \dom(SAT)$ and $y\in \ran(SAT)$.
\end{proof}

In the proof of the next result we revisit a similarity transformation used in \cite{Wa26}.

\begin{theorem}\label{thm:blcrr} Assume \Cref{hyp:fne0p}, \Cref{hyp:gammasus} and that $\dom(C)\cap \ker(C)^\bot \hookrightarrow H_0$ is compact. Let $\lambda \in \R\setminus\{0\}$ and $B(\lambda)$ as \cref{eq:bl}. Then $B(\lambda)$ has closed range.
\end{theorem}
\begin{proof} We introduce
\[
T(\lambda)\coloneqq \iu\lambda \begin{pmatrix}1 & 0&0\\ 0 & 1& 0 \\ 0 & 0 & 1 \end{pmatrix}+
\begin{pmatrix} \iota_0^*\gamma \iota_0 & 0&\iota_0^*\gamma \kappa_0\\ 0 & 0& 0 \\ \kappa_0^*\gamma \iota_0 & 0 & \kappa_0^*\gamma \kappa_0 \end{pmatrix}.
\]
By \Cref{lem:gaminv}, we get that $\iu\lambda+\kappa_0^*\gamma\kappa_0$ is continuously invertible. We compute for 
\begin{align*}
  T_1(\lambda)& \coloneqq 
\begin{pmatrix}1 & 0&-\iota_0^*\gamma\kappa_0 (\iu\lambda+\kappa_0^*\gamma\kappa_0)^{-1}\\ 0 & 1& 0 \\ 0 & 0 & 1 \end{pmatrix}   \text{ and }\\ T_2(\lambda)&\coloneqq  \begin{pmatrix}1 & 0&0 \\ 0 & 1& 0 \\ -(\iu\lambda+\kappa_0^*\gamma\kappa_0)^{-1}\kappa_0^*\gamma\iota_0  & 0 & 1 \end{pmatrix}
\end{align*}
that
\begin{align*}
T_1(\lambda) T(\lambda) T_2(\lambda)  = \iu\lambda \begin{pmatrix}1 & 0&0\\ 0 & 1& 0 \\ 0 & 0 & 1 \end{pmatrix}+
\begin{pmatrix} \iota_0^*\gamma \iota_0-\iota_0^*\gamma\kappa_0 (\iu\lambda+\kappa_0^*\gamma\kappa_0)^{-1} \kappa_0^*\gamma \iota_0 & 0&0\\ 0 & 0& 0 \\ 0 & 0 & \kappa_0^*\gamma \kappa_0 \end{pmatrix}.
\end{align*}
Moreover, we obtain that 
\[T_1(\lambda)  \begin{pmatrix}0 & -\iota_0^*C^*\iota_1 &0\\ \iota_1^*C \iota_0 & 0& 0 \\ 0 & 0 & 0 \end{pmatrix}  T_2(\lambda) = \begin{pmatrix}0 & -\iota_0^*C^*\iota_1 &0\\ \iota_1^*C \iota_0 & 0& 0 \\ 0 & 0 & 0 \end{pmatrix}\]
Thus, as $T_1(\lambda), T_2(\lambda)$ are topological isomorphisms, it follows from \Cref{lem:SAT} that $B(\lambda)$ has closed range if and only if
\begin{multline*}
   T_1(\lambda)B(\lambda)T_2(\lambda)  \\
  = \iu\lambda \begin{pmatrix}1 & 0&0\\ 0 & 1& 0 \\ 0 & 0 & 1 \end{pmatrix}+
\begin{pmatrix} \iota_0^*\gamma \iota_0-\iota_0^*\gamma\kappa_0 (\iu\lambda+\kappa_0^*\gamma\kappa_0)^{-1} \kappa_0^*\gamma \iota_0 & 0&0\\ 0 & 0& 0 \\ 0 & 0 & \kappa_0^*\gamma \kappa_0 \end{pmatrix} \\ +\begin{pmatrix}0 & -\iota_0^*C^*\iota_1 &0\\ \iota_1^*C \iota_0 & 0& 0 \\ 0 & 0 & 0 \end{pmatrix}
\end{multline*}
has closed range. Note that the operator
\[
   \iu\lambda \begin{pmatrix}1 & 0\\ 0 & 1\end{pmatrix}+
\begin{pmatrix} \iota_0^*\gamma \iota_0-\iota_0^*\gamma\kappa_0 (\iu\lambda+\kappa_0^*\gamma\kappa_0)^{-1} \kappa_0^*\gamma \iota_0 & 0\\ 0 & 0\end{pmatrix}
\]
is bounded and, hence, relatively compact with respect to
\[
  \begin{pmatrix}0 & -\iota_0^*C^*\iota_1\\ \iota_1^*C \iota_0 & 0 \end{pmatrix}
\]as the latter operator has compact resolvent. Thus,
\[
r_0 \coloneqq \ran\big( \iu\lambda \begin{pmatrix}1 & 0\\ 0 & 1\end{pmatrix}+
\begin{pmatrix} \iota_0^*\gamma \iota_0-\iota_0^*\gamma\kappa_0 (\iu\lambda+\kappa_0^*\gamma\kappa_0)^{-1} \kappa_0^*\gamma \iota_0 & 0\\ 0 & 0\end{pmatrix}+  \begin{pmatrix}0 & -\iota_0^*C^*\iota_1\\ \iota_1^*C \iota_0 & 0 \end{pmatrix}
\big)
\]
is closed by a standard result of Fredholm theory. Since $\iu\lambda+\kappa_0^*\gamma\kappa_0$ is onto by \Cref{lem:gaminv}, we obtain
\[
   \ran(T_1(\lambda)B(\lambda)T_2(\lambda))=r_0\oplus \ker(C)
\]
is closed, which yields the claim using \Cref{lem:SAT}.
\end{proof}

For $\lambda=0$, the reasoning is different.

\begin{theorem}\label{thm:clr0} Assume \Cref{hyp:fne0p}.  Let $B(0)$ as given in \cref{eq:bl}. Then
\[
B(0)\text{ has closed range} \iff \kappa_0^*\gamma \kappa_0 \text{ has closed range}.
\]
\end{theorem}
Note that the reformulation using the square roots of $\alpha$ and $\beta$ employed in the previous section does not influence the closed range characterisation. This is confirmed next.
\begin{proposition}\label{prop:alpha1} Let $0<c\leq \alpha=\alpha^*,\gamma \in L(H_0)$, $0<c\leq  \beta=\beta^*\in L(H_1)$ for some $c>0$. Then
 $\kappa_0^*\gamma \kappa_0$ has closed range if and only if
\[
{\tilde{\kappa}_0}^*\sqrt{\alpha}^{-1}\gamma\sqrt{\alpha}^{-1} \tilde{\kappa}_0
\]has closed range, where $\tilde{\kappa}_0\colon \ker(\sqrt{\beta}^{-1}C\sqrt{\alpha}^{-1})\hookrightarrow H_0$ is the canonical embedding.
\end{proposition}
\begin{proof} It suffices to show one implication. Thus, assume $\kappa_0^*\gamma \kappa_0$ has closed range.
Note $x\in \ker(\sqrt{\beta}^{-1}C\sqrt{\alpha}^{-1})$ if and only if $\sqrt{\alpha}^{-1}x \in \ker(C)$. Let $y_n = {\tilde{\kappa}_0}^*\sqrt{\alpha}^{-1}\gamma \sqrt{\alpha}^{-1}{\tilde{\kappa}_0}x_n\to y$. Then $(\sqrt{\alpha}^{-1}{\tilde{\kappa}_0}x_n)_n$ is a sequence in $\ker(C)$. Hence, for $z\in \ker(\sqrt{\beta}^{-1}C\sqrt{\alpha}^{-1})$, we deduce
 \[
\langle \sqrt{\alpha}y_n, \sqrt{\alpha}^{-1}z\rangle= \langle y_n,z\rangle =   \langle {\tilde{\kappa}_0}^*{\sqrt{\alpha}^{-1}}\gamma \sqrt{\alpha}^{-1}{\tilde{\kappa}_0}x_n, z\rangle = \langle \gamma \kappa_0\sqrt{\alpha}^{-1}x_n, \kappa_0\sqrt{\alpha}^{-1}z\rangle.
 \]
 Hence, $\kappa_0^*\gamma \kappa_0\sqrt{\alpha}^{-1}x_n$ weakly converges to some $\kappa_0^*\sqrt{\alpha}y$, and as the range of $\kappa_0^*\gamma \kappa_0$ is closed, it is weakly closed and thus we find $\tilde{x} \in \ker(C)$ with $x\coloneqq \sqrt{\alpha} \tilde{x}$ such that $\kappa_0^*\gamma \kappa_0 \sqrt{\alpha}^{-1}x=\kappa_0^*\sqrt{\alpha}y$. In particular, for all $z\in  \ker(\sqrt{\beta}^{-1}C\sqrt{\alpha}^{-1})$, we get
 \begin{multline*}
\langle y,z\rangle =\langle \kappa_0^*\sqrt{\alpha}y,\sqrt{\alpha}^{-1} z\rangle =   \langle\kappa_0^*\gamma \kappa_0 \sqrt{\alpha}^{-1}x, \sqrt{\alpha}^{-1} z\rangle \\ = \langle\gamma \sqrt{\alpha}^{-1}x, \sqrt{\alpha}^{-1} z\rangle =\langle\sqrt{\alpha}^{-1}\gamma \sqrt{\alpha}^{-1}x,  z\rangle.
 \end{multline*}
 As $x\in \ker(C\sqrt{\alpha}^{-1})=\ker(\sqrt{\beta}^{-1}C\sqrt{\alpha}^{-1})$ it follows that
 \[
   y = {\tilde{\kappa}_0}^*\sqrt{\alpha}^{-1}\gamma\sqrt{\alpha}^{-1} \tilde{\kappa}_0 x.\qedhere
 \]
\end{proof}

\begin{lemma}\label{lem:block} Let $H_0,H_1$ be Hilbert spaces.

(a) Let $A\in L(H_0)$ and let $B\colon \dom(B)\subseteq H_1 \to H_0$ as well as $C\colon \dom(C)\subseteq H_0\to H_1$ be both continuously invertible. Then $ \begin{pmatrix}
    A & B \\ C & 0
   \end{pmatrix} $
on its natural domain $\dom(C)\times \dom(B)$ as an operator in $H_0\times H_1$ is continuously invertible and we have
\[
   \begin{pmatrix}
    A & B \\ C & 0
   \end{pmatrix} ^{-1} = 
   \begin{pmatrix}
    0 & C^{-1} \\ B^{-1} & -B^{-1}AC^{-1}
   \end{pmatrix} 
\] 
(b) Let $a\colon \dom(H_0)\subseteq H_0\to H_0$ continuously invertible,  $b\in L( H_1,H_0)$, $c\in L( H_0,H_1)$, $d\in L(H_1)$. Then the following conditions are equivalent:
\begin{enumerate}
\item[(i)]$
      \begin{pmatrix}
    a & b \\ c & d
   \end{pmatrix}$ has closed range;
   \item[(ii)] $d-ca^{-1}b$ has closed range.
\end{enumerate}
\end{lemma}
\begin{proof}
(a) The statement follows by direct computation.

(b) The operators 
\[
  \begin{pmatrix}
    1 & 0 \\ -ca^{-1} & 1
   \end{pmatrix}\text{ and }  \begin{pmatrix}
    1 & -a^{-1}b \\ 0 & 1
   \end{pmatrix}
\]are topological isomorphisms. Hence, $\begin{pmatrix}
    a& b \\ c & d
   \end{pmatrix}$ has closed range if and only if 
   \[
  \begin{pmatrix}
    1 & 0 \\ -ca^{-1} & 1
   \end{pmatrix} \begin{pmatrix}
    a & b \\ c & d
   \end{pmatrix}\begin{pmatrix}
    1 & -a^{-1}b \\ 0 & 1
   \end{pmatrix} = \begin{pmatrix}
    a & 0 \\ 0 & d-ca^{-1}b
   \end{pmatrix}
   \]
   has closed range; as $a$ is continuously invertible, it is onto. Hence, proving (b).
\end{proof}
\begin{proof}[Proof of \Cref{thm:clr0}] First, we apply \Cref{lem:block} (b) to the setting 
\[a= \begin{pmatrix} \iota_0^*\gamma \iota_0 & 0\\ 0 & 0  \end{pmatrix} + \begin{pmatrix}0 & -\iota_0^*C^*\iota_1 \\ \iota_1^*C \iota_0& 0  \end{pmatrix},  b= \begin{pmatrix} \iota_0^* \gamma\kappa_0 \\0\end{pmatrix},  c = \begin{pmatrix} \kappa_0^*\gamma\iota_0 & 0\end{pmatrix}, d = \kappa_0^* \gamma \kappa_0. \] 
Now, consider
\[
 a=  \begin{pmatrix} \iota_0^*\gamma \iota_0 & 0\\ 0 & 0  \end{pmatrix} + \begin{pmatrix}0 & -\iota_0^*C^*\iota_1 \\ \iota_1^*C \iota_0 & 0  \end{pmatrix}.
\]The off diagonal entries of $a$ are continuously invertible and thus $a$ is continuously invertible by \Cref{lem:block} (a). Hence, $B(0)$ has closed range if and only if
\[
d-ca^{-1}b \text{ has closed range.}
\]
Now, by \Cref{lem:block} (a), the top left block operator entry of $a^{-1}$, $(a^{-1})_{11}$,  is $0$. Hence,
\[
   c a^{-1} b = \begin{pmatrix} \kappa_0^*\gamma\iota_0 & 0\end{pmatrix} a^{-1} \begin{pmatrix} \iota_0^* \gamma\kappa_0 \\0\end{pmatrix} = \kappa_0^*\gamma\iota_0 (a^{-1})_{11} \iota_0^* \gamma\kappa_0 =0.
\]
Hence, $B(0)$ has closed range if and only if $d$ has.
\end{proof}

The closedness of the range of the operator $\kappa_0^* \gamma \kappa_0$ appears to be a subject matter requiring particular consideration. If $\gamma$ is strictly positive definite, the claim is plain. For the partially damped situation, there cannot be an abstract argument---even if $\gamma$ was an orthogonal projection, clearly satisfying \Cref{hyp:gammasus}---as the next example shows.

\begin{example}
Let $H_0$ be a Hilbert space and $T\in L(H_0)$ without closed range (as a consequence, $T^*T$ has no closed range).
Then take $\kappa \colon H_0 \to H_0\times H_0, x\mapsto (x,Tx)$ and $p\in L(H_0\times H_0)$ such that $p(x,y) =(0, y)$.  Then $\kappa^*(x,y)= x+T^*y$ and, thus, $\kappa^* p \kappa x = \kappa^* p (x,Tx) =\kappa^*(0,Tx)=T^*Tx$. Thus, $\ran(\kappa^*p\kappa)=\ran(T^*T)$ implying that $\kappa^*p\kappa$ has no closed range.
\end{example}

Next, we consider the kernel of $B(0)$ and its adjoint. The argument needs less attention to the particular block structure.
\begin{theorem}\label{thm:kerneladj} Let $H$ be a Hilbert space and $A\colon \dom(A)\subseteq H\to H$ be densely defined and closed such that $\Re \langle Ax,x\rangle =0$ for all $x\in \dom(A)$. Let $\gamma \in L(H)$ satisfy the assumption in \Cref{hyp:gammasus} with $H_0=H$.

Then $\ker(\gamma+A)=\ker(\gamma-A)=\ker(\gamma)\cap \ker(A)$.
\end{theorem}
\begin{proof}
 It suffices to show that $\ker(\gamma+A)\subseteq \ker(\gamma)\cap \ker(A)$ as the inverse inclusion is self-evident and as $\ker(A)=\ker(-A)$. 
 
 Write $\gamma= \gamma_U + \gamma_{U^\bot}$ as in  \Cref{hyp:gammasus} and let $p$ be the orthogonal projection onto $U$. Note that it follows that $\Re p\gamma=\Re \gamma p = \Re\gamma_U\geq c p$ and $(1-p)\gamma=\gamma_{U^\bot}\geq 0$ for some $c>0$. 
 
Next, let $x\in \ker(\gamma+A)$. Then
 \begin{align*}
     0& = \Re \langle \gamma x +A x,x\rangle \\
     & = \Re \langle \gamma x, x\rangle =\Re \langle \gamma x, px\rangle+\Re \langle \gamma x, (1-p)x\rangle \\
     & \geq c \Re \langle px ,x\rangle + \langle \gamma_{U^\bot} (1-p)x, (1-p)x\rangle\geq c\|px\|^2.
 \end{align*}
 Hence, $x\in \ker(p)$ and, thus, $p\gamma x =\gamma p x=0$.
Moreover, we read off
 \[
 0\geq \langle \gamma_{U^\bot}(1- p) x, (1-p)x\rangle = \|\sqrt{\gamma_{U^\bot}}(1-p)x\|^2, 
 \]
 that is, $(1-p)x \in \ker(\sqrt{\gamma_{U^\bot}})=\ker ({\gamma_{U^\bot}})$.
 It follows that
 \[
    \gamma x = \gamma_U px + \gamma_{U^\bot} (1-p)x = 0
 \]
 and therefore $x\in \ker(\gamma)$. Hence,
 \[
   Ax = \gamma x + A x = 0,
 \]
 and we get $x\in \ker(\gamma)\cap \ker(A)$. 
\end{proof}

We may now provide a proof for \Cref{thm:clrsum}.

\begin{proof}[Proof of \Cref{thm:clrsum}]
 (a) This is \Cref{thm:clr0} and \Cref{prop:alpha1}.
 
 (b) This is \Cref{thm:blcrr}.
 
 (c) This  is \Cref{thm:kerneladj}.
\end{proof}

We conclude this section by providing an abstract counterexample of a $C_0$-semigroup which is strongly stable but not semi-uniformly stable. This reiterates the need for more particular arguments pertaining to the situation at hand, when one was to show semi-uniform stability.

\begin{theorem}\label{thm:counterabstr} Let $C\colon\dom(C)\subseteq H_0\to H_1$ be densely defined and closed with $\dom(C)\cap \ker(C)^\bot\hookrightarrow H_0$ compactly with $\ker(C)$ infinite-dimensional. Then there exists $0\leq \gamma=\gamma^*\in L(H_0)$ such that 
\[
 - \begin{pmatrix} \gamma & 0 \\ 0 & 0 \end{pmatrix} - \begin{pmatrix} 0 & -C^* \\ C & 0 \end{pmatrix}
\]
generates a strongly stable $C_0$-semigroup, which is not semi-uniformly stable.
\end{theorem}
\begin{proof}
 Take any $D\colon\dom(D)\subseteq H_0\to H_1$ densely defined and closed, one-to-one and onto with $\dom(D)\hookrightarrow H_0$ compactly. Let $H_2$ be an infinite-dimensional Hilbert space. Let $\gamma_{ij} \in L(H_j,H_i)$ for $i,j\in \{0,2\}$ so that $\gamma_{ij}^*=\gamma_{ji}^*$ and $\gamma \coloneqq (\gamma_{ij})_{i,j} \in L(H_0\oplus H_2)$ is non-negative. Choose both $\gamma_{00}$ and $\gamma_{22}$ to be one-to-one. Moreover choose $\gamma_{22}$ in such a way that is has no closed range. Then $\gamma$ satisfies \Cref{hyp:gammasus}. Moreover, for all $\lambda\in \R\setminus\{0\}$,
 \[
B_\lambda \coloneqq \Big(\iu\lambda \begin{pmatrix}1 & 0&0\\ 0 & 1& 0 \\ 0 & 0 & 1 \end{pmatrix}+
\begin{pmatrix} \gamma_{00} & 0&\gamma_{02}\\ 0 & 0& 0 \\ \gamma_{20} & 0 & \gamma_{22} \end{pmatrix} + \begin{pmatrix}0 & -D^* &0\\ D & 0& 0 \\ 0 & 0 & 0 \end{pmatrix}\Big)
 \]
 the operator (together with its adjoint) has closed range, by \Cref{thm:blcrr}.
 Next, for all $\lambda\in \R$, we have, by \Cref{thm:counterabstr},
 \begin{align*}
  \ker(B_{\lambda}) & = \ker \begin{pmatrix} \gamma_{00} & 0&\gamma_{02}\\ 0 & 0& 0 \\ \gamma_{20} & 0 & \gamma_{22} \end{pmatrix} \cap \ker \Big(\iu\lambda \begin{pmatrix}1 & 0&0\\ 0 & 1& 0 \\ 0 & 0 & 1 \end{pmatrix}+ \begin{pmatrix}0 & -D^* &0\\ D & 0& 0 \\ 0 & 0 & 0 \end{pmatrix}\Big) \\
  & = \ker \begin{pmatrix} \gamma_{00} & 0&\gamma_{02}\\ 0 & 0& 0 \\ \gamma_{20} & 0 & \gamma_{22} \end{pmatrix} \cap \Big(\ker \Big(\iu\lambda \begin{pmatrix}1 & 0\\ 0 & 1  \end{pmatrix}+ \begin{pmatrix}0 & -D^* \\ D & 0  \end{pmatrix}\Big) \oplus \ker (\iu \lambda )\Big).
 \end{align*}
 Thus, if $\lambda\neq 0$, then 
 \[
   \ker(B_{\lambda}) = \ker \begin{pmatrix} \gamma_{00} & 0&\gamma_{02}\\ 0 & 0& 0 \\ \gamma_{20} & 0 & \gamma_{22} \end{pmatrix} \cap \Big(\ker \Big(\iu\lambda \begin{pmatrix}1 & 0\\ 0 & 1  \end{pmatrix}+ \begin{pmatrix}0 & -D^* \\ D & 0  \end{pmatrix}\Big) \oplus \{0\}\Big).
 \]
 Thus, if $(u,v,w) \in \ker(B_{\lambda})$, we obtain $w=0$ and
 \[
     \gamma_{00} u = 0 \text{ and } \iu \lambda u = D^*v \text{ and } \iu \lambda v = - Du.
 \]
 As $\gamma_{00}$ is one-to-one, we deduce $u=0$ and, hence, $v=0$. 
 
 Note that the same reasoning applies to 
 \[
 \tilde{B}_\lambda = \Big(-\iu\lambda \begin{pmatrix}1 & 0&0\\ 0 & 1& 0 \\ 0 & 0 & 1 \end{pmatrix}+
\begin{pmatrix} \gamma_{00} & 0&\gamma_{02}\\ 0 & 0& 0 \\ \gamma_{20} & 0 & \gamma_{22} \end{pmatrix} - \begin{pmatrix}0 & -D^* &0\\ D & 0& 0 \\ 0 & 0 & 0 \end{pmatrix}\Big) = B_\lambda^*.
\]
Thus, as both $B_\lambda$ and $B_\lambda^*$ are one-to-one and the former has closed range, it follows that $B_\lambda$ is continuously invertible for all $\lambda\neq 0$. 
 
 Next, if $\lambda =0$, then
 \begin{align*}
    \ker(B_0) & = \ker \begin{pmatrix} \gamma_{00} & 0&\gamma_{02}\\ 0 & 0& 0 \\ \gamma_{20} & 0 & \gamma_{22} \end{pmatrix} \cap \Big(\ker \Big(\begin{pmatrix}0 & -D^* \\ D & 0  \end{pmatrix}\Big) \oplus H_2\Big) \\
    & = 
\ker \begin{pmatrix} \gamma_{00} & 0&\gamma_{02}\\ 0 & 0& 0 \\ \gamma_{20} & 0 & \gamma_{22} \end{pmatrix} \cap \Big(\{0\} \oplus H_2\Big).
 \end{align*}
 Hence, $(u,v,w)\in \ker(B_0)$ satisfy $u=v=0$ and
 \[
    \gamma_{02} w = 0 \text{ and } \gamma_{22} w =0.
 \] 
 Again, also $B_0^* = \tilde{B}_0$ is one-to-one. However, by \Cref{thm:clr0} it follows that $B_0$ has no closed range as $\gamma_{22}$ fails to have closed range. Thus, $0\in \sigma(B_0)$ without being an eigenvalue. As a consequence, by \Cref{thm:ABLV}, 
 \[
  -\big(  \begin{pmatrix} \gamma_{00} & 0&\gamma_{02}\\ 0 & 0& 0 \\ \gamma_{20} & 0 & \gamma_{22} \end{pmatrix} + \begin{pmatrix}0 & -D^* &0\\ D & 0& 0 \\ 0 & 0 & 0 \end{pmatrix}\big)
 \]generates a strongly stable semi-group. On the other hand, since $\sigma(B_0)\cap \iu \R=\{0\}$, by \cite[Theorem 1.1]{BD08}, this strongly stable semi-group is not semi-uniformly stable.
 
 Finally, notice that the shape of the operator mentioned in the claim can be obtained with a simply row- and column operation resulting in an operator of the form
 \[
    \begin{pmatrix} \gamma_{00} & \gamma_{02}& 0\\ \gamma_{20} & \gamma_{22} & 0 \\ 0 & 0& 0 \end{pmatrix} + \begin{pmatrix}0 & 0 & -D^*\\ 0 & 0& 0 \\ D & 0 & 0 \end{pmatrix}
 \]
 so that the assertion follows be taking $H_2=\ker(C)$, $D=\iota^*_1C\iota_0$ and $\gamma =  \begin{pmatrix} \gamma_{00} & \gamma_{02}\\ \gamma_{20} & \gamma_{22}  \end{pmatrix}$ and $C = \begin{pmatrix} 0 \\ D \end{pmatrix}$.
\end{proof}

\section{Application to Maxwell's equations with partial damping}\label{sec:apptoMax}

In this section, we apply our abstract results to Maxwell's equation with partial damping. For this, we let $\Omega\subseteq \R^3$ be a connected, open weak Lipschitz domain.
Recall
\[
   \curl \colon \dom(\curl)\subseteq L_2(\Omega)^3 \to L_2(\Omega)^3
\] to
be the distributional curl-operator with maximal domain in $L_2(\Omega)^3$. Note that this operator is densely defined and closed; we put
\[
   \curl_0\coloneqq \curl^*.
\]
By the Picard--Weber--Weck selection theorem (see \cite{Pi84}), $\dom(\curl_0)\cap \ker(\curl_0)^\bot\hookrightarrow L_2(\Omega)^3$ is compact (and, similarly, for $\curl$). Next, take $\varepsilon,\mu \in W^{1}_\infty(\Omega;\C^{3\times 3}_{\textnormal{sym}})$, i.e., symmetric matrix-valued Lipschitz continuous functions satisfying the following positive definiteness requirements:
Assume there exists $c>0$ such that
\[
   \varepsilon(x)=   \varepsilon(x)^*,\mu(x)=\mu(x)^*\geq c\quad(x\in \Omega) 
   \]
   We single out the condition on the damping: Let
\[
  \sigma \in L_\infty(\Omega;\C^{3\times 3})
\]
and assume there exists a measurable $D\subseteq \Omega$ with non-empty interior such that
\[
   L_2(\Omega) = L_2(D)\oplus L_2(\Omega\setminus D)
\]
satisfying 
\[
  \Re \sigma (x)\geq c\quad(x\in D)\text{ and }  \sigma (x)=\sigma(x)^*\geq 0\quad(x\in \Omega\setminus D).
\]

Moreover, we put
\[
A_{\textnormal{Max}}\coloneqq - \begin{pmatrix} \sigma & 0 \\ 0 & 0 \end{pmatrix}-\begin{pmatrix} 0 & -\curl \\ \curl_0 & 0 \end{pmatrix}
\]
on $H=L_2(\Omega)^3\times L_2(\Omega)^3$.

\subsection*{Strong stability}
At first we start with strong stability and consider the following reformulation of solutions of the Maxwell problem. As we will reformulate everything for the original Maxwell system in the next subsection, we present the version for strong stability in the following reduced form.

Define
\begin{equation}\label{eq:Amc} 
A=- \begin{pmatrix}\sqrt{\varepsilon}^{-1} \sigma \sqrt{\varepsilon}^{-1} & 0 \\ 0 & 0 \end{pmatrix}-\begin{pmatrix} 0 & -\sqrt{\varepsilon}^{-1}\curl \sqrt{\mu}^{-1}\\ \sqrt{\mu}^{-1}\curl_0\sqrt{\varepsilon}^{-1} & 0 \end{pmatrix}.
\end{equation}
\begin{theorem}\label{thm:stsMax} Let $(E_0,H_0)\in \ker(A)^\bot$ and let $(E,H)$ be the unique mild solution of
\begin{equation}\label{eq:maxmod1}
\frac{\dd}{\dd t} \begin{pmatrix} {E} \\ {H} \end{pmatrix} = A  \begin{pmatrix} {E} \\ {H} \end{pmatrix},
\end{equation}
subject to $E(0)=E_0$ and $H(0)=H_0$. Then
\[
   \|(E(t),H(t))\|_{L_2(\Omega)^6}\to 0\text{ as }t\to\infty.
\]
\end{theorem}
\begin{remark} \label{rem:kerA}(a)
By \Cref{thm:kerneladj}, we have 
\begin{align*}
  & \ker(A)  \\&= \ker\Big(\begin{pmatrix}\sqrt{\varepsilon}^{-1} \sigma \sqrt{\varepsilon}^{-1} & 0 \\ 0 & 0 \end{pmatrix}\Big)\cap\ker\Big(\begin{pmatrix} 0 & -\sqrt{\varepsilon}^{-1}\curl \sqrt{\mu}^{-1}\\ \sqrt{\mu}^{-1}\curl_0\sqrt{\varepsilon}^{-1} & 0 \end{pmatrix}\Big) \\
   & =  \Big(\ker(\sqrt{\varepsilon}^{-1} \sigma \sqrt{\varepsilon}^{-1})\oplus L_2(\Omega)^3\Big)\cap \Big(\ker(\sqrt{\mu}^{-1}\curl_0\sqrt{\varepsilon}^{-1})\oplus \ker(\sqrt{\varepsilon}^{-1}\curl \sqrt{\mu}^{-1})\Big)\\
   & =  \Big(\ker(\sqrt{\varepsilon}^{-1} \sigma \sqrt{\varepsilon}^{-1})\cap \ker(\sqrt{\mu}^{-1}\curl_0\sqrt{\varepsilon}^{-1})\Big)\oplus \Big(\ker(\sqrt{\varepsilon}^{-1}\curl \sqrt{\mu}^{-1})\Big).
\end{align*}
In particular, 
\[
  \ker(A)^\bot = \overline{\overline{\ran}(\sqrt{\varepsilon}^{-1} \sigma^* \sqrt{\varepsilon}^{-1})+\ran\big(\sqrt{\varepsilon}^{-1}\curl\sqrt{\mu}^{-1}\big)}\oplus \Big(\ran(\sqrt{\mu}^{-1}\curl_0 \sqrt{\varepsilon}^{-1})\Big),
\]where for we used \Cref{rem:Hss}(b) below for the last computation.
Hence, $(\sqrt{\varepsilon}E_0,\sqrt{\mu}H_0)\in \ker(A)^\bot$ if and only if
\[
\varepsilon E_0 \in \overline{\overline{\ran}(\sigma)+\ran(\curl)}\text{ and }\mu H_0 \in \ran(\curl_0).\qedhere
\]
(b) In \cite{NS25}, the condition on the initial data reads as follows
\[
\mu H_0 \in \ker(\dive_0),
\]
where $\dive_0=-\grad^*$ with $\grad\colon H^1(\Omega)\subseteq L_2(\Omega)\to L_2(\Omega)^3$ being the gradient as an operator in $L_2$ with maximal domain. Since $\Omega$ in \cite{NS25} is assumed to  be simply connected, $\ker(\dive_0)\cap \ran(\curl)=\{0\}$ (see \cite{Pi82,PW22}). Hence, $\ker(\dive_0) = \ran(\grad)^\bot = \ran(\curl_0)$ and so the condition for $H_0$ is the same in this situation. The condition for $E_0$ is more particularly formulated in \cite{NS25}, which is possible their due to the more particular geometric set-up. In any case, it is formulated in such a way that initial data are perpendicular to the kernel of the generator, as it is done here.\end{remark}

For the proof of \Cref{thm:stsMax}, we will require a unique continuation principle for Maxwell's equations in order to establish the injectivity statements. For this we recall a corollary of the main result in \cite[Theorem 1.1]{NW12}, see, \cite[Theorem A.2]{SW24} for a precise formulation fitted to our purposes.

\begin{theorem}\label{thm:ucp} Let $\varepsilon,\mu$ be as above and $\Omega$ be connected. Let $E,H\in L_{2;\textnormal{loc}}(\Omega)^3$ and $\lambda \in \R\setminus \{0\}$ be such that 
\[
    \iu \lambda \varepsilon E - \curl H = 0\text{ and }    \iu \lambda \mu H + \curl E = 0.
\]If for some open set $U\subseteq \Omega$, $E|_U=H|_U=0$, then $E=H=0$.
\end{theorem}

\begin{proposition}\label{prop:ucp} Let $A$ be as in \cref{eq:Amc}. Then, for all $\lambda\in \R\setminus \{0\}$, $\iu \lambda - A$ and $\iu \lambda -A^*$ are one-to-one.
\end{proposition}
\begin{proof}
 In both cases, the argument is the same. We put
 \[
    \gamma \coloneqq \begin{pmatrix}\sqrt{\varepsilon}^{-1} \sigma \sqrt{\varepsilon}^{-1} & 0 \\ 0 & 0 \end{pmatrix}
 \]
 and 
 \[
    \tilde{A}_{\pm} \coloneqq \iu\lambda \pm \begin{pmatrix} 0 & -\sqrt{\varepsilon}^{-1}\curl \sqrt{\mu}^{-1}\\ \sqrt{\mu}^{-1}\curl_0\sqrt{\varepsilon}^{-1} & 0 \end{pmatrix}.
 \]
 Then $\iu\lambda -A = \gamma + \tilde{A}_{+}$ and $\iu\lambda -A^* = \gamma^* + \tilde{A}_-$ and note that $\gamma^*$ and $\gamma$ satisfy \Cref{hyp:gammasus} with the same decomposition in $U$ and $U^\bot$, see \Cref{rem:observgamma}(b). In particular, we have $\ker(\gamma)=\ker(\gamma^*)$. Next, note that for all $x\in \dom(\tilde{A}_{\pm})$ we have $\Re \langle \tilde{A}_{\pm}x,x\rangle = 0$. As a consequence, by \Cref{thm:kerneladj}, we get
 \[
   \ker(\iu\lambda -A) = \ker(\gamma)\cap \ker( \tilde{A}_{+})\text{ and }\ker(\iu\lambda -A^*) = \ker(\gamma)\cap \ker( \tilde{A}_{-}).
 \]
Let $(\tilde{E},\tilde{H}) \in \ker(\gamma)\cap \ker( \tilde{A}_{+})$. Then we deduce for $(E,H) = (\sqrt{\varepsilon}^{-1} \tilde{E},\sqrt{\mu}^{-1}\tilde{H}) \in L_{2,\textnormal{loc}}(\Omega)^6$ that
\[
 E \in \ker(\sigma),\text{ that is }E=0 \text{ on }D
\]
and
\[
  (E,H) \in \ker\big(\iu \lambda \begin{pmatrix} \varepsilon & 0 \\ 0 & \mu \end{pmatrix} +\begin{pmatrix} 0 & -\curl \\ \curl_0 & 0 \end{pmatrix}\big).
\]
Since
\[
   \iu \lambda \mu H + \curl E = 0,
\]we infer $H=0$ on $D$. Thus, by \Cref{thm:ucp} and as $D$ has non-empty interior, we infer $E=H=0$. The argument is the same for $\tilde{A}_-$.
\end{proof}

\begin{proof}[Proof of \Cref{thm:stsMax}] The claim follows from \Cref{thm:sts} in combination with \Cref{prop:ucp}.
\end{proof}

\subsection*{Semi-uniform stability}

In this subsection, we ask $\Omega\subseteq \R^3$ in addition to the properties mentioned above to have connected complement (i.e., equivalently, to have connected boundary).

Let $D\subseteq \Omega$ be open with continuous boundary (i.e., equivalently, having the segment property)  assume that $L_2(\Omega)=L_2(D)\oplus L_2(D^\textnormal{c})$  with $D^\textnormal{c}\coloneqq \Omega\setminus \overline{D}$ (i.e., $\partial D$ has zero $3$-dimensional Lebesgue) and that
\[
   \1_{D}[\grad[H_0^1(\Omega)]\subseteq L_2(D)^3\text{ is closed}.
\]Next, let $\tilde{\sigma}\in L_\infty(D)^{3\times 3}$. Assume there exists $c>0$ such that
\[
\Re\tilde{\sigma}(x)\geq c\quad(x\in D).
\]
Define $\sigma(x)\coloneqq \tilde{\sigma}(x)$ for $x\in D$ and $\sigma(x)=0$ for $x\in \Omega\setminus D$. Note that in this case $\sigma$ satisfies \Cref{hyp:gammasus} for $H_0=L_2(\Omega)^3$ and $U=L_2(D)^3$. Our result for semi-unfiorm stability thus reads as follows.

\begin{theorem}\label{thm:strongstabMax} There exists $f\colon [0,\infty)\to[0,\infty)$ with $f(t)\to 0$ as $t\to\infty$ such that for all $E, H \in C[0,\infty;\dom(A_{\textnormal{Max}}))$ with $\varepsilon E, \mu H \in C^1[0,\infty; L_2(\Omega)^3)$ satisfying
\begin{equation}\label{eq:ogmax}
    \frac{\dd}{\dd t}\begin{pmatrix} \varepsilon E \\ \mu H \end{pmatrix} = A_{\textnormal{Max}}\begin{pmatrix}  E \\  H \end{pmatrix}
\end{equation}
and $\mu H(0) \in \ran(\curl_0)$ as well as $\varepsilon E(0) \in \overline{L_2(D)^3 + \ran(\curl)}$,
we have
\[
    \|(E(t),H(t))\|_{L_2(\Omega)^6}\leq f(t) \|A_{\textnormal{Max}}(E(0),H(0))\|_{L_2(\Omega)^6}.
\]
\end{theorem}
\begin{remark}\label{rem:kerAp}
The condition on the initial data is the same as in \Cref{rem:kerA}(a) with the additional information that $\ran(\sigma)=L_2(D)^3$ due to the more restrictive condition on $\sigma$.
\end{remark}
The proof of this theorem is a translation of our abstract finding \Cref{thm:sus} into the present situation. We adopt the setting in \Cref{thm:sus} with the following choices
\[
 \alpha = \varepsilon,\ \beta = \mu,\text{ and }\gamma = \sigma.
\]
Moreover, we put
\[
  C = \curl_0.
\]
The first observation recalls the equivalent formulation as carried out earlier in order to put ourselves into the correct framework.
\begin{lemma}\label{lem:maxmod} Let $E,H$ as in \cref{eq:ogmax}. Then $\tilde{E}\coloneqq \sqrt{\varepsilon} E, \tilde{H}\coloneqq \sqrt{\mu} H$ satisfy
\begin{equation}\label{eq:maxmod}
\frac{\dd}{\dd t} \begin{pmatrix} \tilde{E} \\ \tilde{H} \end{pmatrix} = A  \begin{pmatrix} \tilde{E} \\ \tilde{H} \end{pmatrix},
\end{equation}
where $A=- \begin{pmatrix}\sqrt{\varepsilon}^{-1} \sigma \sqrt{\varepsilon}^{-1} & 0 \\ 0 & 0 \end{pmatrix}-\begin{pmatrix} 0 & -\sqrt{\varepsilon}^{-1}\curl \sqrt{\mu}^{-1}\\ \sqrt{\mu}^{-1}\curl_0\sqrt{\varepsilon}^{-1} & 0 \end{pmatrix}$.
  Moreover, $(\tilde{E},\tilde{H})\in C^1([0,\infty);\dom(A))\cap C^1([0,\infty);L_2(\Omega)^6)$ and $(\tilde{E}(0),\tilde{H}(0))\in \ker(A)^\bot$.
\end{lemma}
\begin{proof}
All the computations have been done already before, see also \cite{Wa26} and \Cref{sec:wp}. The only thing that needs confirmation is the conditions on the initial data. This however has been confirmed in \Cref{rem:kerA} and \Cref{rem:kerAp}.\end{proof}

Next, we aim to show that the conditions of \Cref{thm:sus} are satisfied. 
In the next section, we will show the following theorem.

\begin{theorem}\label{thm:pcrr}
Let $\kappa_0 \colon \ker(\curl_0) \hookrightarrow L_2(\Omega)^3$ be the canonical embedding. Then $\kappa_0^*\sigma\kappa_0$ has closed range.
\end{theorem}

With this theorem, we can conclude the proof of \Cref{thm:strongstabMax}.

\begin{proof}[Proof of \Cref{thm:strongstabMax}] At first, we show that the conditions of \Cref{thm:sus} are satisfied for $A$ from \Cref{lem:maxmod}. For this note that the Picard--Weber--Weck selection theorem yields $\mathcal{C}\coloneqq \dom(\curl_0)\cap \ker(\curl_0)^\bot \hookrightarrow L_2(\Omega)^3$ compactly. It is then not difficult to confirm that also $\tilde{\mathcal{C}}\coloneqq \dom(\sqrt{\mu}^{-1}\curl_0\sqrt{\varepsilon}^{-1})\cap \ker(\sqrt{\mu}^{-1}\curl_0\sqrt{\varepsilon}^{-1})^\bot \hookrightarrow L_2(\Omega)^3$ compactly. Indeed, if $(\tilde{E}_n)_n$ is bounded in $\tilde{\mathcal{C}}$. Then $(\sqrt{\varepsilon}^{-1} \tilde{E}_n)_n$ is bounded in $\mathcal{C}$. Thus, it contains an $L_2(\Omega)^3$-convergent subsequence so that the same subsequence of $(\tilde{E}_n)_n$ is also convergent yielding the claim.  As the strict positivity of (the real part of) $\sigma$ on its range is easy to see, $\sigma$ satisfies \Cref{hyp:gammasus}. Condition (a) in \Cref{thm:sus} follows from \Cref{prop:ucp}. Condition (b) is a consequence of \Cref{thm:pcrr}. By \Cref{thm:sus}, we find $f\colon [0,\infty)\to [0,\infty)$ such that for all $(\tilde{E}_0,\tilde{H}_0)\in \dom(A)\cap \ker(A)^\bot$ and $T$, the $C_0$- semi-group generated by $A$,
\[
  \| T(t) (\tilde{E}_0,\tilde{H}_0)\|\leq f(t) \|(\tilde{E}_0,\tilde{H}_0)\|_{\dom(A)}.
\]
Next, let $(E,H)$ be as in \Cref{thm:strongstabMax}. Then, by \Cref{lem:maxmod} it follows that $(\tilde{E},\tilde{H}) \coloneqq (\sqrt{\varepsilon} E,\sqrt{\mu}H) = T(\cdot)(\sqrt{\varepsilon} E(0),\sqrt{\mu}H(0))$. Hence, for all $t\in [0,\infty)$, we obtain
\begin{align*}
 \| (E(t),H(t))\| & \leq \max\{\|\sqrt{\varepsilon}^{-1}\|,\|\sqrt{\mu}^{-1}\|\} \|(\tilde{E}(t),\tilde{H}(t))\| \\
 & \leq  \max\{\|\sqrt{\varepsilon}^{-1}\|,\|\sqrt{\mu}^{-1}\|\} f(t) \|(\tilde{E}(0),\tilde{H}(0))\|_{\dom(A)} \\
 & \leq C  f(t) \|(E(0),H(0))\|_{\dom(A_{\textnormal{Max}})},
\end{align*}
for some $C>0$ independently of $E$ and $H$, establishing the assertion.
\end{proof}
\section{The proof of the final closed range statement}\label{sec:crs}

Before we go into the details of proving the closed range statement, we recall some well-known facts in Hilbert spaces. Note that we already used this fact in  \Cref{rem:kerA} above.
 \begin{remark}\label{rem:Hss}Let $H$ be a Hilbert space and $V\subseteq H$, $U\subseteq H$ be closed subspaces. Then the following statements are true:
 \begin{enumerate}
 \item[(a)] If, in addition, $U\subseteq V$, then $U^{\bot_{V}}=U^{\bot_{H}}\cap V$.
 \item[(b)] $(U\cap V)^\bot = \overline{U^\bot+V^{\bot}}$.
 \end{enumerate}
The proof of (a) follows by mere definition. For the proof of (b) notice that, as $U\cap V\subseteq U,V$, we obtain $U^\bot + V^{\bot} \subseteq (U\cap V)^{\bot}$. As the latter is a closed and linear subspace, $Y\coloneqq \overline{U^\bot + V^{\bot}}\subseteq (U\cap V)^{\bot}$. Next, take $x\in (U\cap V)^{\bot} \cap  \big(U^\bot + V^{\bot}\big)^\bot$. Then,
\[
   \forall y\in Y\colon \langle y,x\rangle = 0.
\]  In particular, $\langle u_{\bot}, x\rangle = \langle v_{\bot}, x\rangle=0$ for all $u_{\bot} \in U^\bot$ and $v_{\bot}\in V^\bot$. Thus, $x\in U^{\bot\bot} = U$ and $x\in V^{\bot\bot} = V$, that is, $x\in U\cap V$. But also $x\in (U\cap V)^\bot$ implying $x=0$. Hence,
\[
   \big(U^\bot + V^{\bot}\big)^{\bot\bot} \subseteq U\cap V, \text{ and therefore }(U\cap V)^\bot = \overline{U^\bot+V^{\bot}},
\]
as desired.
 \end{remark}

We start off with an abstract finding that helps us guiding to the actual statement. 
\begin{theorem}\label{thm:variational} Let $C\colon \dom(C)\subseteq H_0\to H_1$ be closed and densely defined with closed range, $g \in L(H_1)$ satisfy \Cref{hyp:gammasus} with closed range.
Denote by $\kappa_0\colon \ran(C)\hookrightarrow H_1$ the canonical embedding and by $p\in L(H_1)$ the orthogonal projection onto ${\ran}(g)$. If there exists $d>0$ such that for all $u \in \ran(C)\cap \ker(p)^\bot$ (with the orthogonal complement computed in $\ran(C)$)
\[
    \| u\|_{H_1} \leq d \| pu\|_{H_1},
\]
then $\kappa_0^* g \kappa_0$ has closed range.
\end{theorem}
\begin{proof}To establish the claim, we will use the characterisation of closed range by means of the closed inequality (see, e.g., \cite[pp 96]{Gol06}): $\kappa_0^*g\kappa_0$ has closed range if and only if there exists $s>0$ such that
 \[
    \forall u\in \ran(C)\cap \ker(\kappa_0^*g\kappa_0)^\bot\colon \|u\|\leq s\|\kappa_0^*g\kappa_0 u\|.
 \]
The first part of the proof is concerned with reformulating $ \ran(C)\cap \ker(\kappa_0^*g\kappa_0)^\bot$ and the second establishes the inequality.

Thus, let $x\in H_1$. Note that
 \[
   x \in \ker(\kappa_0^* g \kappa_0) \iff \forall y \in \ran(C)\colon \langle g x,y\rangle = 0.
 \]
 In particular,
 \[
   0=  \Re \langle g x,x\rangle \geq c\langle px,x\rangle = c\|px\|^2.
 \]
 Thus, $px =0$. 
 On the other hand, if $x\in \ker(p)\cap \ran(C)$ we obtain for all $y\in \dom(C)$:
 \[
     \langle g \kappa_0 x,\kappa_0 C y \rangle =      \langle g  x, C y \rangle = \langle  p g p x, C y \rangle =0,
 \]that is $x\in \ker(\kappa_0^*g\kappa_0)$.
 
For the closed range inequality, we may now let $\phi \in \ran(C)\cap \ker(p)^{\bot}$. Then
 \[
 \|\phi\|^2 \leq d \langle p \kappa_0 \phi, p\kappa_0 \phi\rangle=    d \langle p \kappa_0 \phi, \kappa_0 \phi\rangle \leq d\frac{1}{c} \Re  \langle g \kappa_0 \phi, \kappa_0 \phi\rangle \leq d\frac{1}{c}  \|\kappa_0^*g \kappa_0 \phi\|\|\phi\|,
 \]which yields the claim.
 \end{proof}

 Next, we come to the (more general than in the previous section) geometric set-up that will be used throughout this section. Let $\Omega\subseteq \R^d$ be open lying in a slab and assume that
 \[D\subseteq \Omega \text{ is open with having continuous boundary (as a subset of $\R^d$)}\]
and such that $\partial D$ has zero $d$-dimensional Lebesgue measure. As a consequence, \[
L_2(\Omega)=L_2(D)\oplus L_2(D^{\co}), \text{ where }
D^{\co}\coloneqq \Omega\setminus \overline{D}.
\] Let $\tilde{\sigma} \in L_\infty(D)^{3\times 3}$ be with $\Re \tilde{\sigma}$ being strictly positive definite uniformly on $D$ and denote by $\sigma$ its zero-extension to the whole of $\Omega$. Define $\grad_0 \colon H_0^1(\Omega)\subseteq L_2(\Omega) \to L_2(\Omega)^d$ the usual gradient with Dirichlet boundary conditions. By Poincar\'e's inequality, as $\Omega$ is contained in a slab, $\ran(\grad_0)\subseteq L_2(\Omega)^3$ is closed. Next, we provide some space decompositions needed in the following.
 
 \begin{proposition}\label{prop:odecomp} Let $\omega\subseteq \Omega$ be open. Then
 \[
     \ran(\grad_0)\cap L_2(\omega )^d = \grad[H_0^1(\omega)]\oplus \mathcal{H}(\omega).
 \] 
for some closed subspace $\mathcal{H}(\omega)\subseteq \ker(\dive_\omega)\subseteq L_2(\omega)^d$
 \end{proposition}
 \begin{proof}
 As $\ran(\grad_0)$ is a closed subspace in $L_2(\Omega)^d$ and, by zero extension, so $L_2(\omega )^d\subseteq L_2(\Omega)^d$ is closed. Since $\Omega$ is contained in a slab, so is $\omega$ and, hence, by Poincar\'e's inequality, $\grad[H_0^1(\omega)]$ is closed. Moreover, since $H_0^1(\omega)\subseteq H_0^1(\Omega)$ (by zero extension) and $\grad\phi \in L_2(\omega)^d$ for every $\phi\in H_0^1(\omega)$, we infer that $\grad[H_0^1(\omega)]\subseteq      \ran(\grad_0)\cap L_2(\omega )^d $. Thus, $\mathcal{H}(\omega)$ is then the $L_2(\Omega)^d$-orthogonal complement of $\grad[H_0^1(\omega)]$ in $  \ran(\grad_0)\cap L_2(\omega )^d$. In particular, it consists of vector fields supported on $\omega$ only, establishing the assertion as $\ran(\grad_{\omega,0})^\bot=\ker(\dive_{\omega})$.
 \end{proof}
 From the proof of \Cref{prop:odecomp} we single out the following definition. For any open $\omega\subseteq \Omega$ we put
 \[
    \mathcal{H}(\omega) \coloneqq \grad[H_0^1(\omega)]^{\bot_{L_2(\omega)^d}}\cap \ran(\grad_0) \subseteq \ker(\dive_\omega)\subseteq L_2(\omega)^d.
 \]
 \begin{remark} Note that $\mathcal{H}(\omega)$ consists of gradient fields the divergence of which vanishes on $\omega$. Thus, the corresponding potential is \emph{harmonic} on $\omega$. If $d=3$, the vector field itself is also \emph{harmonic} as it belongs to the intersection of $\ker(\curl_\omega)$ and $\ker(\dive_\omega)$.
 \end{remark}
 The application of the present particular setting in the context of \Cref{thm:variational} is provided next. We recall $D^{\co}\coloneqq \Omega\setminus \overline{D}$.
 \begin{theorem}\label{thm:concrete} Denote by $\kappa_0\colon \ran(\grad_0)\hookrightarrow L_2(\Omega)^d$ the canonical embedding. If there exists $\alpha >0$ such that for all $q\in \ran(\grad_0)\cap \ker(\dive_{D^{\co}})\cap \mathcal{H}(D^{\co})^{\bot}$
 \[
    \| q\|_{L_2(\Omega)^d} \leq \alpha \|q\|_{L_2(D)^d},
 \]then $\kappa_0^*\sigma \kappa_0$ has closed range.
 \end{theorem}
 \begin{proof}
 It is clear that $\1_{D}$, the multiplication by the characteristic function of the set $D$, is the orthogonal projection $p$ in \Cref{thm:variational} when setting $C=\grad_0$, $g=\sigma$. Thus, it suffices to establish
 \begin{multline*}
    (\ran(C)\cap \ker(p))^\bot = (\ran(\grad_0)\cap L_2(D^{\co})^d)^\bot\\ = \big(\grad[H_0^1(D^{\co})]\oplus \mathcal{H}(D^{\co})\big)^{\bot} =\ker(\dive_{D^{\co}})\cap \mathcal{H}(D^{\co})^{\bot},
 \end{multline*}
 where $\dive_{D^{\co}}$ is the adjoint of $\grad$ restricted to $H_0^1(D^{\co})$ as an operator in $L_2(D^{\co})$,
 which follows from \Cref{prop:odecomp}.
 \end{proof}
 
 \begin{remark}
 Note that the techniques we process in the following to arrive at the decisive estimate (see \Cref{thm:decinequ}) are heavily inspired by the rationale proving the decisive estimate in \cite{PPTW21}. Here we were able to remove some distractions in \cite{PPTW21} related to the willingness of obtaining more explicit formulas and the three-dimensional setting in that reference that conveniently allowed for such computations in the first place. Moreover, some geometric conditions could be lifted and replaced by \Cref{hyp:ep} below.
 \end{remark}
In order to prove the desired closed range statement, we assume that $D$ is a relative extension domain inside $\Omega$ in the following sense. For convenience, we recall the assumptions on $\Omega$ and $D$ also in this hypothesis.
\begin{hypothesis}\label{hyp:ep} Let $\Omega\subseteq \R^d$ be open and contained in a slab and $D\subseteq \Omega$ open with continuous boundary. $\partial D$ is supposed to have $d$-dimensional Lebesgue measure $0$. Assume that 
$\1_{D}[\ran(\grad_0)] \subseteq L_2(D)^d$ is a closed subspace.
\end{hypothesis}
We assume \Cref{hyp:ep} to be in effect up until the end of this section.
\begin{remark}
(a) Let $K \subseteq H^1(D)$ be a closed subspace. Then $\grad[K]\subseteq L_2(D)^d$ is closed. Indeed, since $D$ has continuous boundary, $H^1(D)$ embeds compactly into $L_2(D)$ and thus so does $K$. It follows that $\grad$ restricted to $K$ has closed range.

 (b) If $D$ is compactly contained in $\Omega$ and an \textbf{extension domain}, that is, there exists a continuous extension operator $E\colon H^1(D)\to H^1(\R^d)$ in the way that $Eu|_D=u$ for all $u\in H^1(D)$. Then, for any smooth $\phi \in C_c^\infty(\Omega)$ such that $\phi=1$ on $\overline{D}$, we get that $E_0 \coloneqq \phi E$ is still an extension operator mapping into $H_0^1(\Omega)$. This shows that $\1_{D}[\ran(\grad_0)] = \grad[H^1(D)]$, which is closed by (a).
 
(b) For the next two statements, we introduce
\[
    H^1_{\partial} (D) \coloneqq \overline{\{\phi|_{D}; \phi \in C_c^\infty(\R^d), \dist(\spt\phi,\partial D\cap \partial\Omega)>0 \}}^{H^1(D)}.
\]
If $\Omega$ is Lipschitz then, using \cite[Theorem 7.10]{ACSVV15} where it is shown for the $C^1$-case,
\[
  H^1_0(\Omega) = \{ \phi \in H^1(\Omega); \phi = 0 \text{ on }\partial \Omega\}.
\]
Similarly, if $\partial \Omega\cap \partial D$ is regular enough, then
\[
   H^1_{\partial}(D)  = \{ \phi \in H^1(D); \phi = 0 \text{ on }\partial \Omega\cap \partial D \}.
\]
Note that as $H^1_{\partial}(D)$ is a closed subspace of $H^1(D)$, we get $\grad[H^1_{\partial}(D)]$ is closed by (a).

(c) Let $\partial \Omega\subseteq \partial D$ be a connected and closed subset and both $D$ and $\Omega$ be Lipschitz. Then, 
any $\phi\in H^1_{\partial}(D)$ admits an extension $u$ to $H^1(\R^d)$. Since, by (b),
\[
   H^1_0(\Omega) = \{ u \in H^1(\Omega); u=0 \text{ on }\partial\Omega\},
\]
$u|_{\Omega} \in H_0^1(\Omega)$. Thus,
\[
   \1_{D}[\grad [H_0^1(\Omega)]] = \grad [H^1_{\partial\Omega}(D)].
\]
The latter is closed by (b).

(d) Assume that $\partial D\cap \partial\Omega$ is Lipschitz. Then both conditions (b) and (c) can be summarised as follows. Assume that any $u\in H^1_{\partial}(D)$ admits an extension to $H^1_0(\Omega)$. Then
\[
   \1_{D}\ran(\grad_0) = \grad[H^1_{\partial}(D)]
\]
and the latter space is closed by (b). Indeed, let $u\in H^1_{\partial}(D)$. Then, by assumption, we find $v\in H_0^1(\Omega)$ such that $v|_{D}=u$ and $\grad u =\1_{D}\grad v$. On the other hand, if $v\in H^1_0(\Omega)$, then $u\coloneqq \1_{D} v \in H^1(D)$ and $u = 0 $ on $\partial D\cap \partial\Omega$, which by (b) implies that $u \in H^1_\partial(D)$ and hence, $\1_D\grad v = \grad u \in \grad[H^1_\partial(D)]$, as claimed.
\end{remark}
We introduce the following spaces
\begin{align*}
  H_2 &\coloneqq \ran(\grad_0)\cap \big(\ker(\dive_{D})\cap \mathcal{H}(D)^{\bot_{L_2(D)^d}}\oplus \ker(\dive_{D^{\co}})\cap \mathcal{H}(D^{\co})^{\bot_{L_2(D^{\co})^d}}\big);\\
  H_3& \coloneqq \grad_{D} [\ker(\dive_{D}\grad_{D})\cap \{\phi; \grad_{D} \in \mathcal{H}(D)^\bot\}].
\end{align*}
Our strategy of proof is aligned with the arguments in \cite{PPTW21}. We start off with our version of \cite[Lemma 4.17]{PPTW21}.
\begin{proposition}\label{prop:H3closed} $H_3\subseteq L_2(D)^d$ is a closed subspace and for $U\in H_2$ we have $\1_{D} U\in H_3$. Moreover, $\1_D U \in \1_D\ran(\grad_0)$.
\end{proposition}
\begin{proof}
Clearly, $H_3\subseteq L_2(D)^d$. Let $(\phi_n)_n$ in $\ker(\dive_{D}\grad_{D})\cap \{\phi; \grad_{D} \in \mathcal{H}(D)^\bot\}$ such that
 \[
    \grad_D \phi_n \to \psi \in L_2(D)^d.
 \]
 As $\ran(\grad_D)$ is closed, we find $\phi \in H^1(D)$ such that $\psi = \grad_D \phi$. Next, $\grad_D \phi_n \in \ker(\dive_D)$ implying $\grad_D\phi \in \ker(\dive_D)$ as $\ker(\dive_D)$ is closed; that is, $\phi \in \ker(\dive_D\grad_D)$. Similarly, as $\grad_D \phi_n \in \mathcal{H}(D)^\bot$, we infer $\grad_D \phi \in  \mathcal{H}(D)^\bot$, finally establishing closedness of $H_3$. 
 
 Next, let $U\in H_2$. Then, as $H_2\subseteq \ran(\grad_0)$, we find $v\in H^1_0(\Omega)$ such that $U=\grad_0 v$. In particular, $v|_{D} \in H^1(D)$  and we obtain $\grad (v|_{D}) = \1_D \grad_0 v = \1_D U$. By the definition of $H_2$, we obtain the claim.
\end{proof}
We introduce the (well-defined) operator
\[
   T\colon H_2\to \tilde{H}_3, U\mapsto \1_D U
\]
$\tilde{H}_3\coloneqq H_3\cap \1_D[\ran(\grad_0)]$. By \Cref{hyp:ep} and \Cref{prop:H3closed}, $\tilde{H}_3$ is a closed subspace. The next statement is our version of \cite[Lemma 4.18]{PPTW21}.
\begin{lemma}\label{lem:T11} $T$ is one-to-one.
\end{lemma}
\begin{proof}
 Let $U\in \ker(T)$. Then $U=0$ on $\1_D$. As $U\in \ran(\grad_0)$, we find $\psi\in  \mathcal{H}(D^{\co})$ and $u\in H_0^1(D^{\co})$ such that
 \[
     U = \psi + \grad_{D^{\co},0} u.
 \]
 Since $U\in H_2$, we have $\psi =0$. Hence,
 \[
    U =\grad_{D^{\co},0}u \in \ker(\dive_{D^{\co}}),
 \]
 that is, $U=0$ on $D^{\co}$ by continuous invertibility of $\dive_{D^{\co}}\grad_{D^{\co},0}$.
\end{proof}
Our version of \cite[Lemma 4.19]{PPTW21} reads as follows.
\begin{lemma}\label{lem:Tonto} $T$ is onto.
\end{lemma}
\begin{proof} Sei $W\in \tilde{H}_3$; that is, $w\in \ker(\dive_D)\cap \mathcal{H}(D)^\bot$ and we find $w\in H^1(D)$ such that
\[
    W = \grad_D w
\]
As $W\in  \tilde{H}_3$ there exists $\phi_0 \in H^1_0(\Omega)$
\[
 w = \1_{D}\phi_0.
\]
In particular, we get
\[
\grad_D w = \1_D\grad_0\phi_0.
\]
Let $\pi_{\co}$ be the orthogonal projection onto $\ran(\grad_{D^{\co},0})\subseteq L_2(D^{\co})^d$. By definition, there exists $\theta \in H^1_0(D^{\co})$ such that
\[
   \grad_{D^{\co},0}\theta = -\pi_{\co} \1_{D^{\co}}\grad_0 \phi_{0}
\]
Define
\[
  {\phi} \coloneqq \phi_0 +\theta \in H^1_0(\Omega).
\]
Then
\begin{align*}
  \1_{D^{\co}} \grad_0 \phi &  = \1_{D^{\co}}\grad_0 \phi_0 +\grad_{D^{\co},0}\theta \\
  & = \1_{D^{\co}}\grad_0 \phi_0 -\pi_{\co} \1_{D^{\co}}\grad_0 \phi_{0} \\
  & = (1-\pi_{\co} ) \1_{D^{\co}}\grad_0 \phi_0 \in \ker(\dive_{D^{\co}}).
  \end{align*}
  Next, 
  \[
      \mathcal{H}(D^{\co}) \subseteq \ran(\grad_0)
  \]
  by construction (see \Cref{prop:odecomp}) is a closed subspace. 
  Hence, define
  \[
     U\coloneqq \pi_{ \mathcal{H}(D^{\co})^{\bot(\ran(\grad_0))}} \grad_0 \phi.
  \]
  Hence, 
  \[
      U \in  \mathcal{H}(D^{\co})^{\bot(\ran(\grad_0))} =  \mathcal{H}(D^{\co})^{\bot(L_2(\Omega)^d)}\cap \ran(\grad_0) = ( \mathcal{H}(D^{\co})^{\bot}\oplus L_2(D)^d)\cap \ran(\grad_0).
  \]
  Next,
  \[
    U -\grad_0 \phi \in  \mathcal{H}(D^{\co}) \subseteq \ker(\dive_{D^{\co}})
  \]
  Thus, $U-\grad_0\phi =0$ on $D$ and
  \[
   \1_{D^{\co}} U =   \1_{D^{\co}}(U-\grad_0\phi) +   \1_{D^{\co}}\grad_0\phi \subseteq \ker(\dive_{D^{\co}}).
  \]
  Next,
  \begin{align*}
    \1_{D} U & = \1_{D} \grad_0 \phi \\
    & = \1_{D} \grad_0 \phi_0 + \1_{D} \grad_0 \theta \\
    & = \1_{D} \grad_0 \phi_0 = \grad_D w = W \in \ker(\dive_D)\cap \mathcal{H}(D)^\bot.\qedhere
  \end{align*}
\end{proof}
The next statement is our variant of \cite[Proposition 4.16]{PPTW21}.
\begin{theorem}\label{thm:Tinv} There exists $c_0>0$ such that
\[
    \|U\|_{L_2(\Omega)^d} \leq c_0 \|\1_D U\|_{L_2(\Omega)^d}
\]
for all $U\in H_2$.
\end{theorem}
\begin{proof}
The proof follows from the closed graph theorem, realising that both $H_2$ and $\tilde{H}_3$ are Hilbert spaces and that $T$ is a continuous one-to-one, onto mapping. 
\end{proof}
Finally, we present a version of \cite[Lemma 4.20]{PPTW21} here. For this note that the $\curl$-term is absent due to the higher-dimensional set-up treated here and the fact that the inequality is conditioned to $\ran(\grad_0)$, which in three dimensions results in a vanishing $\curl$-term. In fact, the whole point of the simplification possible due to the abstract \Cref{thm:clr0}, is that we are in a position to actually let go of the $\curl$-term from the outset.
\begin{theorem}\label{thm:decinequ} Assume \Cref{hyp:ep}. Let $H_0\coloneqq \big(\ker(\dive_{D^{\co}})\cap  \mathcal{H}(D^{\co})^\bot\oplus L_2(D)^d\big) \cap \ran(\grad_0)$. Then there exist $c>0$ such that for all $U\in H_0$ we have
\[
    c\|U\|_{L_2(\Omega)^d} \leq \|1_{D} U\|.
\]
\end{theorem}
\begin{proof} Consider $\tilde{H}_0 \coloneqq \ker(\dive_{D^{\co}})\cap  \mathcal{H}(D^{\co})^\bot\oplus L_2(D)^d $. Then $\tilde{H}_0^\bot = \ran(\grad_{D^{\co},0})\oplus  \mathcal{H}(D^{\co})$.  Define $H_1\coloneqq \ran(\grad_{D,0})\oplus \mathcal{H}(D)\subseteq L_2(D)^d$. By assumption, 
\[
\ran(\grad_{D,0})\oplus \mathcal{H}(D) \subseteq \ran(\grad_0).
\] Moreover, recall
\[
 H_2 = \ran(\grad_0)\cap \big(\ker(\dive_{D})\cap \mathcal{H}(D)^{\bot_{L_2(D)^d}}\oplus \ker(\dive_{D^{\co}})\cap \mathcal{H}(D)^{\bot_{L_2(D^{\co})^d}}\big);
\]
We get
\begin{align*}
   \tilde{H}_0 & = \ker(\dive_{D^{\co}})\cap  \mathcal{H}(D^{\co})^\bot\oplus L_2(D)^d \\
   & = \ker(\dive_{D^{\co}})\cap  \mathcal{H}(D^{\co})^\bot\oplus \big(\ran(\grad_{D,0})\oplus \mathcal{H}(D) \oplus \ker(\dive_{D})\cap \mathcal{H}(D)^\bot\big) \\
   & = \big(\ker(\dive_{D})\cap \mathcal{H}(D)^{\bot_{L_2(D)^d}}\oplus \ker(\dive_{D^{\co}})\cap \mathcal{H}(D^{\co})^{\bot_{L_2(D^{\co})^d}}\big) \oplus  \ran(\grad_{D,0})\oplus \mathcal{H}(D).
\end{align*}
In particular, since $H_1\subseteq \ran(\grad_0)$, we infer
\[
   H_0 = H_1 \oplus H_2.
\]
Further note that $U = U_1 +U_2\in H_0 = H_1 \oplus H_2$ implies $U_1= \1_D U_1$.

Hence, by the above theorem,
\[
    \|U\|_{L_2(\Omega)^d}^2 = \|U_1\|^2+\|U_2\|^2 \leq \|U_1\|_{L_2(D)^d}^2+c\|U_2\|_{L_2(D)^d}^2 \leq \tilde{c} \|\1_{D}U\|^2_{L_2(\Omega)^d}.\qedhere
\]
\end{proof}
\begin{proof}[Proof of  \Cref{thm:pcrr}] By assumption on $\Omega$ in \Cref{sec:apptoMax}, the complement of $\Omega$ is connected. Thus, by \cite{PW22} or \cite{Pi82}, we obtain $\ker(\curl_0)=\ran(\grad_0)$. By  
\Cref{thm:decinequ} in conjunction with \Cref{thm:concrete}, we obtain the assertion.
\end{proof}

\section{Conclusion}

In this paper we have revisited statements for strong and semi-unform stability for hyperbolic type equations. Our main motivation and application was drawn from Maxwell's equations with partial damping provided by the conductivity term. Compared to other results, we managed to reduce the amount of geometric assumptions to a rather minimal set of requirements. Moreover, occasional assumptions on self-adjointness or even realness of the conductivity could be removed. It appears to be an \textbf{open problem} to characterise the functional analytic condition in \Cref{hyp:ep} in terms of geometric properties of the underlying domain $\Omega$ and the domain of the damping $D$.

\section*{Data Availability \& Conflict of Interest}

Data Availability: Not applicable. No datasets were generated or analyzed in this study.

Conflict of Interest: The author declares that he has no competing interest.

\end{document}